\documentclass[12pt,a4paper]{amsart}

\usepackage{geometry}

\usepackage{amsmath, amsthm, amssymb, mathabx}
\usepackage[pdftex,bookmarks=true]{hyperref}
\usepackage{color}

\newtheorem{theorem}{Theorem}[section]
\newtheorem{proposition}{Proposition}[section]
\newtheorem{lemma}[theorem]{Lemma}

\newtheorem{definition}[theorem]{Definition}

\def\C{{\mathbb C}}  
\def\S{{\bf S}} 
\def\E{{\textbf{E}}}

\def\Z{{\mathbb Z}} 

\def\R{{\mathbb R}}

\def\P{{\bf P}}

\def\Q{{\bf Q}}
\def\I{{\bf I}}

\def\i{{\imath}}

\def\T{{\bf T}}
\def\Z{{\mathbb Z}}

\def\I{{\bf I}}
\def\J{{\bf J}}

\def\<{\left<}
\def\>{\right>}
\def\endproof{$\Box$}
\def\dist{{\rm dist}}

\newcommand {\rea}{\mathbb{R}}
\newcommand {\ma}{\mathcal{M}}

\newcommand {\BBZ}{\mathbb{Z}}

\def\i{{\bf i}}
\def\e{{\bf e}}
\newcommand {\size}{\mathrm{size}}
\newcommand {\sgn}{\mathop{\mathrm{sgn}}}

\newcommand {\calG}{\mathcal{G}}
\newcommand {\dis}{\mathop{\mathrm{dist}}}

\newcommand {\calJ}{\mathcal{J}}

\author{Yen Do \qquad Richard Oberlin \qquad Eyvindur A. Palsson}
\thanks{Y.D. is supported in part by NSF grants DMS--1201456 and DMS--1521293.}
\thanks{R.O. is supported in part by NSF Grant DMS-1068523.}
 
\title[Estimates for ergodic bilinear  averages]{Variation-norm and fluctuation estimates for  ergodic bilinear averages}
\date{\today}

\address{Department of Mathematics,
The University of Virginia, Charlottesville, VA 22904-4137, USA}
\email{yendo@virginia.edu}

\address{Department of Mathematics,
Florida State University,
Tallahassee, FL 32306-4510, USA}
\email{roberlin@math.fsu.edu}

\address{Department of Mathematics and Statistics,
    Williams College,  Williamstown, MA 01267,   USA}
\email{ eap2@williams.edu}

\begin{document}

\begin{abstract}{For any dynamical system, we show that  higher variation-norms for the sequence of  ergodic bilinear averages of two functions satisfy a large range of bilinear $L^p$ estimates. It follows  that, with probability one, the number of fluctuations along   this sequence  may grow at most polynomially with respect to (the growth of) the underlying scale.    These results strengthen previous works of Lacey and Bourgain where almost surely convergence of the sequence was proved (which is equivalent to the qualitative statement that the number of fluctuations is finite at each scale).  Via  transference, the proof reduces to  establishing new bilinear $L^p$ bounds for  variation-norms of truncated bilinear operators on $\R$, and the main new ingredient of the proof of these bounds is a  variation-norm extension of  maximal Bessel inequalities of Lacey and  Demeter--Tao--Thiele.}
\end{abstract}

\maketitle
 
\section{Introduction}
Let $T$ be an invertible bi-measurable measure-preserving transformation on a complete probability space $(X,\Omega, \mu)$. Given two measurable functions $f_1, f_2$ on $X$, we consider their ergodic bilinear averages, namely
\begin{eqnarray*}
M_k[f_1, f_2](x)  &=& \frac1k \sum_{n=0}^{k-1} f_1(T^n x)  f_2(T^{-n}x)   \ \ (k=1,2,\dots) \ \ .
\end{eqnarray*}
It was shown by Bourgain in \cite{bourgain1990} that if  $f_1,f_2\in L^\infty(X)$ then $(M_k[f_1,f_2](x))_{k\ge1}$  is convergent  for $\mu$-almost every $x\in X$. Thanks to a bilinear maximal function estimate of Lacey \cite{lacey2000},  Bourgain's result remains valid for $(f_1, f_2)\in L^{p_1}\times L^{p_2}$ for every $(p_1,p_2,q)$ satisfying
\begin{eqnarray}\label{e.BHTrange}
\frac{1}{q} = \frac{1}{p_1} + \frac{1}{p_2},\  \frac{2}{3} < q < \infty
, \  \ 1 < p_1,p_2 \leq \infty \ \ , 
\end{eqnarray}
and this has been regarded as a bilinear analogue of the classical Birkhoff ergodic theorem. A similar result also holds for a variant of $M_k$ (namely the ergodic bilinear Hilbert transform), see Demeter \cite{demeter2007} and Demeter--Tao--Thiele \cite{dtt2008}. 

Our aim in this paper is to further demonstrate that   the sequence $M_k[f_1,f_2](x)$, $k\ge 1$,  converges rapidly. To formulate a consequence of our estimates, we recall the notion of fluctuations of a given sequence $(a_1,a_2,\dots)$. Given a scale $\lambda>0$,  the number of fluctuations in $(a_k)$ with respect to this scale is  the largest number $\ell$ such that there exists $\ell$ disjoint intervals
\begin{eqnarray*}
[n_1, m_1), [n_2, m_2), \dots [n_\ell, m_\ell)
\end{eqnarray*}
with the following properties: for every $1\le j \le \ell$ it holds that  $|a_{m_j}-a_{n_j}|\ge 1/\lambda$. It follows from the Cauchy criteria that  $(a_k)$ is convergent if any only if it has a finite number of fluctuations at every (finite) scale. Thus results of \cite{bourgain1990,lacey2000} could be interpreted as saying that: for almost every $x\in X$, at every scale, the number of fluctuations along $M_k(f,g)(x)$ is finite. It turns out that this number grows at most polynomially as $\lambda \to \infty$.

\begin{theorem}\label{t.main-fluctuations} Assume that $p_1,p_2,q$ satisfying \eqref{e.BHTrange}. Then there exists $R<\infty$ such that  for every $f_1\in L^{p_1}$ and $f_2 \in L^{p_2}$ the following holds: for almost every $x\in X$ the number of fluctuations in the sequence $(M_k[f_1,f_2](x))_{k\ge 1}$ at any scale $\lambda>0$ is bounded above by $O(\lambda^R)$, where the implicit constant is uniform over $\lambda$ but could depends on $x$ and $f_1,f_2$.
\end{theorem}

 For an interesting discussion about applications of fluctuation estimates in ergodic theory, we refer the readers to Avigad--Rute \cite{ar12+} (cf. Kovac \cite{kovac2014}).

Theorem~\ref{t.main-fluctuations} is an immediate consequence of Theorem~\ref{t.main-varnorm-discrete} below, which provides a more quantitative estimate. 
To formulate this result, we recall the notion of variation-norm. Given  $\Omega\subset \R$ and  $a: \Omega \to \C$, let its $r$-variation norm be
\begin{eqnarray*}\|a(t)\|_{V^r_t(\Omega)} &:=& \sup_{n,  N_0<\dots<N_n} (|a(N_0)|^r + \sum_{j=1}^n |a(N_j)-a(N_{j-1})|^r)^{1/r} \ \ ,
\end{eqnarray*}
in the sup we require $N_j\in \Omega$ for every $j$. We also use the semi-norm variant $\widetilde V^r$ defined similarly without the first term $|a(N_0)|^r$.

\begin{theorem}\label{t.main-varnorm-discrete} Assume that $p_1,p_2,q$ satisfying \eqref{e.BHTrange}. Then there exists $R<\infty$ such that the following holds for every $r>R$: 
\begin{eqnarray*}
\|M_k[f_1,f_2](x)\|_{L^q_x(V^r_k)} &\lesssim& \|f_1\|_{p_1} \|f_2\|_{p_2}
\end{eqnarray*}
\end{theorem}

Via a modification of standard transference arguments (which we will detail in Section~\ref{s.transference}), Theorem~\ref{t.main-varnorm-discrete} follows from $L^p$ estimates for bilinear singular integrals,  Theorem~\ref{t.main-varnorm} below. To formulate the result, we fix some notations.

Given $K:\R\to \C$ sufficiently nice, consider the  bilinear  operator with kernel $K$
\begin{eqnarray} \label{e.B-def}
B[f_1,f_2](x) = \int_{\mathbb R} f_1(x + y) f_2(x - y) K(y)\ dy \ \ ,
\end{eqnarray}
which is \emph{a priori} well-defined for Schwarz functions $f_1$ and $f_2$. For any $t>0$ let $B_t$ be the  bilinear operator with kernel $t^{-1} K(t^{-1}y)$.  

We will be interested in  $K:\R\to \C$ such that the following properties hold uniformly over $\xi \ne 0$:
\begin{eqnarray}
\label{e.K-condition1}
|\widehat{K}(\xi)| &\lesssim& \min (1,\frac{1}{|\xi|}) \ \ , \\
\label{e.K-condition2} |\frac{d^n}{d\xi^n} \widehat{K}(\xi)| &\lesssim_n& \min(\frac{1}{|\xi|^{n-1}},\frac{1}{|\xi|^{n+1}}), \qquad n \ge 1  \ \ .
\end{eqnarray}
We will in fact work with $K$ where \eqref{e.K-condition2} holds for $1\le n \le n_0$, here $n_0$ is some given large number; now the implicit constants are allowed to depend on $n_0$. In this case, we will say that $K$ satisfies \eqref{e.K-condition1} and \eqref{e.K-condition2} up to order $n_0$.

\begin{theorem} \label{t.main-varnorm}
Assume that $p_1,p_2,q$ satisfies \eqref{e.BHTrange} and $r>2$. Then there exists $n_0$ finite such that if $K$ satisfies \eqref{e.K-condition1} and \eqref{e.K-condition2} up to order $n_0$ then
\begin{eqnarray*}
 \|B_t(f_1, f_2)(x)\|_{L^q_x(V^r_t)} &\lesssim& \|f_1\|_{L^{p_1}(\R)} \|f_2\|_{L^{p_2}(\R)}  \ \ ,
\end{eqnarray*}
where the implicit constant may depend on $n_0$ and on the implicit constants of \eqref{e.K-condition1} and \eqref{e.K-condition2} for $1\le n \le n_0$.
\end{theorem} 


Comparing Theorem~\ref{t.main-varnorm} with Theorem~\ref{t.main-varnorm-discrete}, it can be seen that there is a discrepancy between the two ranges $r>2$ and $r>R$. With the current transference techniques, it seems that to get the  range $r>2$ for Theorem~\ref{t.main-varnorm-discrete}  one would need a version of Theorem~\ref{t.main-varnorm} that accommodates  rougher $K$'s, such as $K(y)=1_{|y|\le 1}$, which would be an interesting open problem left for future studies. In fact, in our transference argument we also prove a weaker version of Theorem~\ref{t.main-varnorm} for this particular $K$ where instead of $r>2$ we only have $r>R$ for some finite $R$, see Theorem~\ref{t.average-varnorm}.

Our proof of Theorem~\ref{t.main-varnorm} could be viewed as a variation-norm extension of Lacey's proof of the boundedness of the bilinear maximal function  in  \cite{lacey2000}, although we will follow more closely the expositions in  Demeter--Tao--Thiele \cite{dtt2008} and Demeter \cite{demeter2007}. The main new ingredient of the proof (compared to \cite{lacey2000,demeter2007,dtt2008}) is a variation-norm extension of maximal Bessel inequality for phase plane projections, which in turn relies on variation-norm estimates for Fourier projection operators associated with a collection of frequencies. Maximal estimates for these multi-frequency projection operators were introduced in Bourgain \cite{bourgain1990}, and  variation-norm estimates for smooth multi-frequency Fourier projections were also considered in \cite{not2010}. In our context, it turns out that we need variation-norm estimates for sharp multi-frequency Fourier projections, similar to the original settings considered by Bourgain. On the other hand, $L^2$ bounds would be sufficient for our purpose, and these estimates are proved in Theorem~\ref{t.bourgain-varnorm} by adapting an argument in \cite{not2010}.

We mention some closely related works in addition to \cite{bourgain1990,lacey2000,dtt2008,demeter2007}. A  dyadic version of Theorem~\ref{t.main-varnorm} was considered in  our previous work \cite{dop2013} (which in turn is an adaptation of Thiele \cite{thiele2001} to the variation-norm setting).  The method of proof in Demeter \cite{demeter2007} relies on a weaker version of Theorem~\ref{t.main-varnorm} where the variation-norms are replaced by finitary oscillation norms, which were also used by  Demeter--Lacey--Tao--Thiele \cite{dltt2008} (see also Demeter \cite{demeter2009,demeter2012}, Nazarov--Oberlin--Thiele \cite{not2010}) to improve the $L^p$ ranges in the Bourgain return time theorem.  For a nice introduction to variation-norm estimates in harmonic analysis, see Jones--Seeger--Wright \cite{jsw2008}.  The time-frequency analysis framework used in our proof originated from Lacey--Thiele's proof of the boundedness of the bilinear Hilbert transform \cite{lt1999,lt2000}.

\subsection{Outline of the paper}
 In Section~\ref{s.transference} we detail the transference argument that deduces Theorem~\ref{t.main-varnorm-discrete} from Theorem~\ref{t.main-varnorm}. In Section~\ref{s.shortlong} we discuss how a short-long decomposition of the variation-norm leads to a reduction of Theorem~\ref{t.main-varnorm} to two sub-theorems, which respectively treat the contribution of the long-jumps  and the contribution of the short-jumps. The proof of these Theorems will use restricted weak-type interpolation methods, which we recall in Section~\ref{s.lindual}. In Section~\ref{s.terms} we recall standard terminologies in time-frequency analysis, 
which will be used  in Section~\ref{s.discretization} to describe some wave packet representation for the operators underlying the long-jump and short-jump contributions. Some old and new auxiliary estimates will be recalled and proved in Section~\ref{s.auxiliary}, Section~\ref{s.varnorm-multiplier}, Section~\ref{s.size}. In Section~\ref{s.varnorm-bessel} we prove a new variation-norm extension of the maximal Bessel inequalities of Lacey \cite{lacey2000} and Demeter--Tao--Thiele \cite{dtt2008}, which will be used in  Section~\ref{s.sizelemma} and Section~\ref{s.long-concluding} to prove the desired estimates for the contribution of the long-jumps. In Section~\ref{s.short-concluding} we briefly discuss the needed cosmetic changes that could be applied (to the treatment of the long-jump contribution) to get the desired estimates for the short-jump contributions.

\subsection{Notational convention}

Given an interval $I$, we let $c(I)$ denote the center of the interval, and for each constant $C>0$ we let $CI$ denote the dilate of $I$ around its center by the factor $C$. We will use $\e$ and $\i$ to refer to the numbers $\exp(1)$ and $\sqrt{-1}$ respectively, leaving their non-boldfaced counterparts free for other purposes.

For every interval $I$ let $\widetilde \chi_I(x) = (1+(\frac{x-c(I)}{|I|})^2)^{-2}$.  

For each $s \geq 1$ we let $\ma^s$ denote the $L^s$ Hardy-Littlewood maximal operator 
\begin{eqnarray*}
\ma^s[f](x)  &:=& \sup_{R} \left(\frac{1}{2R} \int_{x - R}^{x + R}|f(y)|^s\ dy\right)^{1/s}
\end{eqnarray*}
and abbreviate $\ma := \ma^1.$

Throughout the paper we let $\mathcal F$ denote the Fourier transform 
\begin{eqnarray}\label{e.fourier-transform}
\widehat h(\xi) \quad \equiv \quad \mathcal F[h(.)](\xi) &:=& \int_{\R} {\bf e}^{-\i 2\pi \xi x} h(x)dx \ \ .
\end{eqnarray}
Note that with this normalization we have 
\begin{eqnarray*} h(x) &=& \int_{\R} {\bf e}^{\i 2\pi x\xi} \widehat h(\xi)d\xi \ \ .
\end{eqnarray*}

\section{The transference argument} \label{s.transference}

In this section we deduce Theorem~\ref{t.main-varnorm-discrete}  from Theorem~\ref{t.main-varnorm} using a variant of standard transference arguments in \cite{bourgain1990,dtt2008}.  Our first step is to show that the continuous version Theorem~\ref{t.main-varnorm-discrete} holds, namely
\begin{theorem}\label{t.average-varnorm}For every $t>0$ let $S_t$ denote the following operator
\begin{eqnarray*}
S_t[f_1,f_2](x) &=& \frac 1 t \int_0^t f_1(x+t) f_2(x-t)dt \ \ .
\end{eqnarray*}
Then for every $(p_1,p_2,q)$ satisfying \eqref{e.BHTrange} there exists $R<\infty$ such that for every $r>R$ it holds that
\begin{eqnarray}\label{e.average-varnorm}
\|S_t[f_1,f_2](x)\|_{L^q_x(V^r_t)} &\lesssim& \|f_1\|_{p_1} \|f_2\|_{p_2} \ \ .
\end{eqnarray}
\end{theorem}

\proof[Proof of Theorem~\ref{t.average-varnorm}]
If $(p_1,p_2,q)$ that satisfies \eqref{e.BHTrange} we let $n_0=n_0(p_1,p_2)$ be the constant required in Theorem~\ref{t.main-varnorm}.

Fix $r$ below. We divide the proof into two steps. 

\underline{Step 1:} Let $R_0=2(1+(n_0+1)\frac{u_0}{u_0-1})$ where $u_0=\min(p_1,p_2,2q)>1$.  We first show that for $r>R_0$ it holds that
\begin{eqnarray}\label{e.average-jump}
\|\sup_{\lambda>0} \lambda N(S_t,\lambda)^{\frac 1 r}\|_{L^q_x} \lesssim \|f_1\|_{p_1} \|f_2\|_{p_2}
\end{eqnarray}

Clearly, we may find   $1<u<u_0$ and $r_0>2$ such that $r> r_0(1+(n_0+1)\frac{u}{u-1})$.  For brevity, let $n_1= (n_0+1)\frac{u}{u-1}$.

Now, for each $0<\alpha \le 1/2$ let $K_\alpha$ be a $C^\infty$ function supported in $[0,1]$ such that $1_{\alpha \le y \le 1-\alpha}  \le K_\alpha(y)\le 1$, we may construct $K$ such that $|K_\alpha^{(n)}| \lesssim \alpha^{-n}$ for any $n\ge 1$.

It is clear that for any $n\ge 0$ and $k\ge 0$ we have
\begin{eqnarray*}
\frac{d^n}{d\xi^n} \widehat K_\alpha(\xi) &\lesssim_{n,k}& \alpha^{-k} (1+|\xi|)^{-k}  \ \ .
\end{eqnarray*}
Therefore  $\alpha^{n_0+1} K_\alpha$ satisfies the assumptions \eqref{e.K-condition1} and \eqref{e.K-condition2} up to order $n_0$ (we emphasize that the implicit constants are independent of $\alpha$). Let $B_{t,\alpha}$ denote the bilinear operator with kernel $\frac 1 t K_\alpha( \frac y t)$. It follows that for any $r_0>2$ we have
\begin{eqnarray}\label{e.Balpha-varnorm}
\alpha^{n_0+1}  \|B_{\alpha,t}[f_1,f_2](x)\|_{L^q_x(V^{r_0}_t)} \lesssim \|f_1\|_{p_1} \|f_2\|_{p_2} \ \ .
\end{eqnarray}

Let $S^\ast$ denote the positive maximal version of $S_t$, namely 
\begin{eqnarray*} S^\ast[f_1,f_2](x) &=& \sup_{t>0} S_t[|f_1|,|f_2|](x) \ \ .
\end{eqnarray*}
By the bilinear maximal estimate of Lacey, it holds that
\begin{eqnarray*}
\|S^\ast[f_1,f_2]\|_{L^q} \lesssim \|f_1\|_{p_1}\|f_2\|_{p_2} \ \ .
\end{eqnarray*}
Let $u>1$ be such that $u<\min(p_1,p_2,2q)$, then applying the above estimate for the triple $(\frac{p_1}u, \frac {p_2}u, \frac {q}u)$ we obtain
\begin{eqnarray}\label{e.Lacey}
\|S^\ast[|f_1|^u, |f_2|^u]^{1/u}\|_{L^q} \lesssim \|f_1\|_{p_1}\|f_2\|_{p_2}
\end{eqnarray}

Now, for brevity in the following we understand that $S_t=S_t[f_1,f_2](x)$, $B_{\alpha,t}=B_{\alpha,t}[f_1,f_2](x)$, $S^\ast = S^\ast[f_1,f_2](x)$, and $S^{\ast,u}=S^\ast[|f_1|^u, |f_2|^u](x)^{1/u}$. 

Given any sequence (or functions) $\{a(t), t\in \Omega\}$,  let $N(a,\lambda)$ be the number of fluctuations with respect to scale $1/\lambda$, i.e. the largest $k$ such that there exists a sequence of $k$ disjoint intervals $[N_0,N_1)$, \dots, $[N_{k-1}, N_k)$, where each $N_j\in \Omega$ and furthermore
$|a_{N_j}-a_{N_{j-1}}|>\lambda$ for every $1\le j \le k$.

For any $t>0$, using Holder's inequality we have
\begin{eqnarray*}
|S_t  - B_{\alpha,t}| &\le&   (2 \alpha)^{(u-1)/u} S^{\ast,u} \ \ .
\end{eqnarray*}
Let $\beta  := (2\alpha)^{(u-1)/u}$, we have
\begin{eqnarray}\label{e.BtoS}
N(S_t, 3 \beta S^{\ast,u}) \le N(B_{\alpha,t}, \beta S^{\ast,u})
\end{eqnarray}
here the fluctuation counts are used with respect to the $t$ variable. Using the basic estimate $\lambda N(a,\lambda)^{1/r_0} \lesssim  \|a\|_{V^{r_0}}$ and using \eqref{e.Balpha-varnorm},  for every $r_0>2$ we have
\begin{eqnarray*}&& \Big\|\beta^{1+n_1} S^{\ast,u} \cdot N\Big(S_t, 3\beta S^{\ast,u} \Big)^{1/r_0} \Big\|_{L^q_x} =\\
&=& (2\alpha)^{1+n_0}\Big\| \beta S^{\ast,u} \cdot N\Big(S_t, 3\beta S^{\ast,u} \Big)^{1/r_0} \Big\|_{L^q_x}   \lesssim \\
&\lesssim& \alpha^{1+n_0}  \|B_{\alpha,t}[f_1,f_2](x)\|_{L^q_x(V^{r_0}_t)} \lesssim \|f_1\|_{p_1} \|f_2\|_{p_2} \ \ .
\end{eqnarray*}
Using the Holder inequality and \eqref{e.Lacey}, it follows that
\begin{eqnarray*}
&&\|\beta S^{\ast,u}  \cdot N(S_t, 3\beta S^{\ast,u})^{\frac 1{r_0(1+n_1)}} \|_{L^{q}_x}  \lesssim \\
&\lesssim& \|S^{\ast,u}\|_{L^q_x}^{\frac{n_1}{1+n_1}} \cdot \Big\|\beta^{1+n_1} S^{\ast,u} \cdot N\Big(S_t, 3\beta S^{\ast,u} \Big)^{\frac 1{r_0}} \Big\|_{L^q_x}^{\frac 1{1+n_1}}
\end{eqnarray*}
therefore
\begin{eqnarray}\label{e.betaS}
\|\beta S^{\ast,u}  \cdot N(S_t, 3\beta S^{\ast,u})^{\frac 1{r_0(1+n_1)}} \|_{L^{q}_x}  &\lesssim&\|f_1\|_{p_1} \|f_2\|_{p_2}
\end{eqnarray}
We note that this estimate holds for any $0\le \beta \le 1$.  Letting $\beta=2^{-k}/3$, $k\ge 0$, and using the triangle inequality it follows that
\begin{eqnarray*}
\|\sum_{k\ge 0} 2^{-(1+\epsilon)k} S^{\ast,u} N \Big(S_t, 2^{-k} S^{\ast,u}\Big)^{\frac 1{r_0(1+n_1)}}\|_{L^q_x} 
&\lesssim& \|f_1\|_{p_1}\|f_2\|_{p_2} \ \ .
\end{eqnarray*}
Since $N(S_t,\lambda)=0$ for $\lambda \ge 2 S^{\ast}$, and since $S^{\ast,u}\ge S^\ast$, it follows that
\begin{eqnarray*}\| (S^{\ast,u})^{-\epsilon} \sup_{\lambda>0} \lambda^{1+\epsilon}  N(S_t,\lambda)^{\frac 1 {r_0(1+n_1)}}\|_{L^q_x} 
&\lesssim& \|f_1\|_{p_1} \|f_2\|_{p_2} \ \ .
\end{eqnarray*}
Using Holder's inequality and using \eqref{e.Lacey}, we obtain
\begin{eqnarray*}
\|  \sup_{\lambda>0} \lambda N(S_t,\lambda)^{\frac 1 {r_0(1+n_1)(1+\epsilon)}}\|_{L^q_x} &\lesssim& \|f_1\|_{p_1} \|f_2\|_{p_2}
\end{eqnarray*}
therefore by choosing $\epsilon$  small so that $r>(1+\epsilon)r_0(1+n_1)$ we obtain \eqref{e.average-jump}.

\underline{Step 2:} We now prove \eqref{e.average-varnorm}; the argument below is similar to an argument in \cite{dmt2012}. We plan to use bilinear Marcinkiewicz interpolation: given each $(p_1,p_2,q)$ satisfying \eqref{e.BHTrange} we may let $R$ to be the largest $R_0$ of the exponents associated with any four rectangular weak-type endpoints. Let $r>R$, then we could use \eqref{e.average-jump} at all of these weak-type endpoints. By monotone convergence it suffices to show that for any increasing sequence of measurable functions $(N_k)$ it holds that
\begin{eqnarray*}
\|(\sum_{k} |S_{N_k}-S_{N_{k-1}}|^r)^{1/r}\|_q &\lesssim& \|f_1\|_{p_1} \|f_2\|_{p_2} \ \ .
\end{eqnarray*}
Let $T[f_1,f_2]=(\sum_{k} |S_{N_k}[f_1,f_2]-S_{N_{k-1}}[f_1,f_2]|^r)^{1/r}$. By bilinear interpolation it suffices to prove the weak-type estimate
\begin{eqnarray*}
\lambda |\{x: T[f_1,f_2](x)>\lambda\}|^{\frac 1 q} &\lesssim&  \|f_1\|_{p_1} \|f_2\|_{p_2}
\end{eqnarray*}
with uniform implicit constants over $\lambda>0$.  By scaling symmetries and dilation symmetry of $S_t$, we may assume $\lambda=\|f_1\|_{p_1}=\|f_2\|_{p_2}=1$. Let  
\begin{eqnarray*}
E &=& \{x: \sup_{k} |S_{N_k}-S_{N_{k-1}}| > 1\}
\end{eqnarray*}
 Clearly $|E| \le \|N(S_t,1)\|_q \lesssim 1$. For $x\not\in E$, we estimate $T[f_1,f_2](x)$ by considering level sets for $|S_{N_k}-S_{N_{k-1}}|$ (as a function of $k$) and obtain:
\begin{eqnarray*} T[f_1,f_2](x)^q  &\lesssim&  \Big(\sum_{j\ge 0} 2^{-jr} N(S_t,2^{-j})\Big)^{q/r}\\
&\lesssim_{\epsilon,q}& \sum_{j\ge 0} 2^{-(1-\epsilon)q j} N(S_t,2^{-j})^{q/r} 
\end{eqnarray*}
Therefore by the Chebysheff inequality we obtain
\begin{eqnarray*}
|\{x: T[f_1,f_2](x)>1\}|^{1/q} &\lesssim&
|E|^{1/q}+ |\{x\not\in E: T[f_1,f_2](x) >1\}|^{1/q} \\
&\lesssim& 1+ (\int \sum_{j\ge 0} 2^{-(1-\epsilon)q j} N(S_t,2^{-j})^{q/r}dx)^{1/q}  \ \ .
\end{eqnarray*}
Using \eqref{e.average-jump} for $\widetilde r=\frac r{1+\delta}$ where $\delta>0$ is sufficiently small so that $\widetilde r>R$, we have
\begin{eqnarray*}
\int  N(S_t,2^{-j})^{q/r}dx &=& \|N(S_t,2^{-j})^{1/\widetilde r}\|_{L^{q/(1+\delta)})x}^{q/(1+\delta)} \lesssim  2^{jq/(1+\delta)}  
\end{eqnarray*}
Therefore
\begin{eqnarray*}
|\{x: T[f_1,f_2](x) >1\}|^{1/q} &\lesssim&
 1+ (\sum_{j\ge 0} 2^{-(1-\epsilon)qj} 2^{jq/(1+\delta)})^{1/q} \quad \lesssim \quad 1
\end{eqnarray*}
by choosing $\epsilon>0$ sufficiently small depending on $\delta$ (which in turn depends on $r$ and $R$). This completes the proof of \eqref{e.average-varnorm}.
\endproof

We now transfer Theorem~\ref{t.average-varnorm} to the integers. Fix $r>R$. We'll show that for any two sequences $f_1(n)$ and $f_2(n)$  indexed by $\mathbb Z$ it holds that
\begin{eqnarray*}
\|\widetilde M_{k}[f_1, f_2](n)\|_{L^q_n(V_k^r)} &\lesssim& \|f_1\|_{\ell^{p_1}(\Z)} \|f_2\|_{\ell^{p_2}(\Z)} \ \ ,\\
\widetilde M_k[f_1,f_2](n) &:=& \frac 1 k \sum_{m=0}^{k-1} f_1(n+m)f_2(n-m) \ \ .
\end{eqnarray*}
To see this, let $S_t$ be the bilinear  operator defined in \eqref{e.B-def} with kernel $t^{-1} 1_{0<y<t}$, where $t>0$. We extend $f_1$ and $f_2$ from $\Z$ to $\mathbb R$ by letting:
 
\begin{itemize}
\item [(i)] $F_1(x)=f_1(n)$ if there exists $n\in \Z$ such that $|x-(n+1/2)| <1/3$, and $F_1(x)=0$ otherwise;
\item[(ii)] $F_2(x)=f_2(n)$  if there exists $n\in \Z$ such that $|x-(n-1/2)| < 1/3$, and $F_2(x)=0$ otherwise.
\end{itemize}
 Let $n\in \Z$ and $x\in [n-\frac 1 6, n + \frac 1 6]$. Then for any $m \in \mathbb Z$  it holds that
\begin{eqnarray*}
\int_{m<y<m+1}  F_1(x+y)F_2(x-y)dy  
&=& \frac {2}3  f_1(n+m) f_2(n-m)  
\end{eqnarray*}
Thus for any $k\ge 0$ we have
\begin{eqnarray*}
S_k(F_1,F_2)(x) &=&  \frac1{k}\int_{0<y<k} F_1(x+y) F_2(x-y)dy =  \frac {2}3 \widetilde M_k[f_1, f_2](n) \ \ ,
\end{eqnarray*}
and consequently
\begin{eqnarray*}
\|\widetilde M_{k}[f_1,f_2](n)\|_{V_k^r} 
&\lesssim& \inf_{x\in [n-\frac 1 6, n+\frac 1 6]} \|S_{k}(F_1, F_2)(x)\|_{V^r_k} \ \  .
\end{eqnarray*}
It follows that
\begin{eqnarray*}
\| \widetilde M_{k}[f_1,f_2](n)\|_{L^q_n(V^r_k)}
 &\lesssim&  \|S_{t}(F_1,F_2)(x)\|_{L^q_x(V^r_t)} \ \ ,
\end{eqnarray*}
and using Theorem~\ref{t.average-varnorm} we can bound the right hand side by
\begin{eqnarray*}
&\lesssim& \|F_1\|_{L^{p_1}(\R)}\|F_2\|_{L^{p_2}(\R)}  = C' \|f_1\|_{\ell^{p_1}(\Z)} \|f_2\|_{\ell^{p_2}(\Z)} \ \ .
\end{eqnarray*}

Our next step is to transfer the result on $\Z$ to a more general setting. Let  $T$ be a measure-preserving transformation on a complete probability space  $(X,\Omega,\mu)$. Let $f$ and $g$ be given. 

Fix a large integer $N$, which we will send to $\infty$ later. All implicit constants below   are independent of $N$ and $x$.

For fixed $x$, let $M(f,g,N)(x)$ be the $r$-variation norm of the finite sequence indexed by $0\le k \le N$:
\begin{eqnarray*}
\frac 1 k  \sum_{0\le m \le k-1} f(T^m x) g(T^{-m}x)  \qquad  .
\end{eqnarray*}

Note that for every $0\le n \le N$ the value of $M(f,g,N)(T^n x)$  depends only on $f(T^mx)$ and $g(T^m x)$ with $|m|\le 2N$. Thus, using the $\mathbb Z$-result, it follows that
\begin{eqnarray*}
\sum_{|n|\le N} |M(f,g,N)(T^n x)|^q &\lesssim& \big(\sum_{|m|\le 2N} |f(T^mx)|^{p_1}\big)^{q/p_1} \big(\sum_{|m|\le 2N} |g(T^mx)|^{p_2}\big)^{q/p_2}
\end{eqnarray*}
Integrating over $x\in X$ and using the H\"older inequality, we obtain
\begin{eqnarray*} &&\sum_{|n|\le N} \int_X |M(f,g,N)(T^n x)|^q d\mu(x) \quad \lesssim\\
&\lesssim&  \Big(\sum_{|m|\le 2N} \int_X |f(T^mx)|^{p_1} d\mu(x)\Big)^{q/p_1} \Big(\sum_{|m|\le 2N} \int_X |g(T^mx)|^{p_2} d\mu(x)\Big)^{q/p_2} \ \ .
\end{eqnarray*}
Using the fact that $T$ is bi-measure preserving on $(X,\mu)$, we obtain
\begin{eqnarray*}
\| M(f,g,N)\|_{L^q(X,\mu)} &\lesssim& \|f\|_{L^{p_1}(X,\mu)} \|g\|_{L^{p_2}(X,\mu)} \ \ , 
\end{eqnarray*}
and by sending $N \to \infty$ we obtain  the conclusion of Theorem~\ref{t.main-varnorm-discrete}. This completes the transference argument, and the rest of the paper is devoted to the proof of Theorem~\ref{t.main-varnorm}. We'll assume that $K$ satisfies \eqref{e.K-condition1} and \eqref{e.K-condition2} up to some large order that may depend on $p_1,p_2,q,r$. We will also free the symbol $S_t$ which could be used in the future for  different purposes.

\section{Separation of  short   and long jumps}\label{s.shortlong}

For any function $a(t)$ on $\R$ it is not hard to see that \begin{eqnarray*}
\|a(t)\|_{V^r_t} &\lesssim& \|a(t)\|_{S_t} + \|a(2^n)\|_{V^r_n(\Z)} \ \ , \\
\|a(t)\|_{S_t} &:=& (\sum_{n\in \Z} \|a(t)\|^2_{\widetilde V^2_t([2^{n},2^{n+1}])})^{1/2}
\end{eqnarray*}
Applying this estimate to $a(t)=B_t[f_1,f_2](x)$, the proof of Theorem~\ref{t.main-varnorm} is divided into two parts: the first part handle the long-jumps (i.e. $\|a(2^n)\|_{V^r_n(\Z)}$) and the second part handles the short jumps (i.e. $\|a(t)\|_{S_t}$).


\begin{theorem}\label{t.long-jump} For any $r>2$ and $p_1,p_2,q$ satisfying \eqref{e.BHTrange} it holds that
\begin{eqnarray*}
\|B_{2^n}(f_1, f_2)(x)\|_{L^q_{x\in \R}(V^r_{n \in \Z})} &\lesssim_{p_1,p_2,r}& \|f_1\|_{p_1} \|f_2\|_{p_2}
\end{eqnarray*}
\end{theorem}

\begin{theorem}\label{t.short-jump} Assume  that $p_1,p_2,q$ satisfy \eqref{e.BHTrange} and $r>2$. Assume that $K_s,1\le s\le 2$ is a family of kernels such that $K_s$ satisfies \eqref{e.K-condition1}, \eqref{e.K-condition2} up to a high order $n_0=n_0(p_1,p_2,q,r)$, and furthermore
\begin{eqnarray}\label{e.Ks-condition}
\widehat K_s(\xi) \lesssim |\xi| \ \ , \ \ \xi \ne 0 \ \ ,
\end{eqnarray}
and the implicit constants are uniform over $1\le s \le 2$. Then  it holds that
\begin{eqnarray*}
\|\int_1^2 (\sum_{n\in \Z} |\int f_1(\cdot +y) f_2(\cdot - y) 2^{-n} K_s(2^{-n}y)dy|^2)^{1/2} ds\|_{q} &\lesssim& \|f_1\|_{L^{p_1}(\R)} \|f_2\|_{L^{p_2}(\R)} \ \ .
\end{eqnarray*}
\end{theorem}
Theorem~\ref{t.long-jump} immediately takes care of the long-jump component of the variation norm $\|B_t[f,g]\|_{V^r_t}$. Below we deduce the desired estimate for the short jump component from Theorem~\ref{t.short-jump}.

We first note that if $a(t)$ is differentiable then using $\|.\|_{\widetilde V^2}\le \|.\|_{\widetilde V^1}$ we obtain
\begin{eqnarray*}
\|a(t)\|_{\widetilde V^2_t[2^n,2^{n+1}]} \lesssim  \|a'(t)\|_{L^1_t[2^n,2^{n+1}]} 
&=& 2^n \|a'(2^n s)\|_{L^1_s[1,2]} \ \ .
\end{eqnarray*}
Therefore using Minkowski's inequality we have
\begin{eqnarray*}
\|a(t)\|_{S_t} &\lesssim& \|(\sum_{n\in\Z} 2^{2n} |a'(2^n s)|^2)^{1/2}\|_{L^1_s[1,2]}\\
&\lesssim& \|(\sum_{n\in\Z} |(2^{n}s) a'(2^n s)|^2)^{1/2}\|_{L^1_s[1,2]}
\end{eqnarray*}

We plan to apply the estimate to $a(t)=B_t[f_1, f_2](x)$ where $x$ is fixed. Let $h(y) = -(K(y) + y K'(y))$, or equivalently $\widehat h(\xi) =  \xi\frac{d}{d\xi} \widehat K(\xi)$.
 Let $H_t$ be the bilinear singular integral with kernel $t^{-1} h(t^{-1}y)$. Then
\begin{eqnarray*}
t\frac {d}{dt} B_t[f_1, f_2](x) &=& H_t[f_1, f_2](x) \ \ , 
\end{eqnarray*}
therefore
\begin{eqnarray*}
\|B_t[f_1, f_2](x)\|_{S_t} &\lesssim& \|(\sum_{n\in \Z} |H_{2^n s}[f_1, f_2](x)|^2)^{1/2}\|_{L^1_s[1,2]}
\end{eqnarray*}
We may write $H_{2^n s}[f_1,f_2](x) =  \int f_1(x+y) f_2(x-y) 2^{-n} K_s(2^{-n} y) dy$ with $K_s(y):=s^{-1} h(s^{-1} y)$, and it is not hard to see that $K_s$ satisfies \eqref{e.K-condition1}, \eqref{e.K-condition2}, \eqref{e.Ks-condition} uniformly in $s \in[1,2]$. Thus, the desired estimates for the short jump component of $B_t[f_1,f_2]$ follows from Theorem~\ref{t.short-jump}.

In the rest of the paper we prove Theorem~\ref{t.long-jump} and Theorem~\ref{t.short-jump}. We will use the restricted weak-type interpolation approach of \cite{mtt2002}, which will be discussed in the next section.

\section{Linearization and interpolation}

\subsection{Linearization}\label{s.lindual}
For each $x$ consider a measurable function $L:\R \to \Z_+$ the set of positive integers, and two sequences of measurable functions: a non-decreasing integer valued sequence $(k_n(x))_{n=0}^{L(x)}$ and a sequence $(a_n(x))_{n=1}^{L(x)}$ such that $\sum_{n\ge 0} |a_n(x)|^{r'} \le 1$. Then an appropriate choice of $L$ and such sequences guarantees that
\begin{eqnarray*}
\|B_{2^k} (f_1,f_2)(x)\|_{V^r_k(\Z)} &\le& 2 \sum_{n=1}^{L(x)} \Big(B_{2^{k_n}}[f_1,f_2](x) - B_{2^{k_{n-1}}}[f_1,f_2](x)\Big)a_n(x) \ \ .   
\end{eqnarray*}

Similarly, for each $s\in [1,2]$ we may find a sequence of measurable functions $(d_n(s,x))_{n=-\infty}^\infty$ such that $\sum_{n} |d_n(s,x)|^2 \le 1$, and
\begin{eqnarray*} 
&&(\sum_{n\in \Z} |\int f_1(x +y) f_2(x - y) 2^{n} K_s(2^{n}y)dy|^2)^{1/2} \\
&\le& 2 \sum_{n \in \Z} \int f_1(x +y) f_2(x - y) 2^{n} K_s(2^{n}y)dy\ d_n(s,x)
\end{eqnarray*}

The desired estimates in Theorem~\ref{t.long-jump} and Theorem~\ref{t.short-jump} follow from certain restricted-weak type estimates for the following tri-linear forms, which we will discuss in the next section.
\begin{eqnarray*}
\Lambda_{long}(f_1,f_2,f_3) &=& \<\sum_{n=1}^{L} \Big(B_{2^{k_n}}[f_1,f_2] - B_{2^{k_{n-1}}}[f_1,f_2]\Big)a_n, f_3\>\\
\Lambda_{short}(f_1,f_2,f_3) &=& \int_1^2 \Lambda_{short,s}(f_1,f_2,f_3) ds\\
\Lambda_{short,s}(f_1,f_2,f_3) &=&   \<\sum_{n \in \Z} \int f_1(\cdot +y) f_2(\cdot - y) 2^{-n} K_s(2^{-n}y)dy d_n(s,\cdot), f_3\>  \  \ .
\end{eqnarray*}

\subsection{Restricted weak-type interpolation}

For any $G\subset \R$ with finite Lebesgue  measure, we say that $H\subset G$ is a minor subset if $|H|\le |G|/2$. 

Let  $\alpha=(\alpha_1,\alpha_2,\alpha_3)\in \R^3$ be such that $\alpha_1+\alpha_2+\alpha_3=1$ and at most one $\alpha_j$ could be negative. We say that a tri-linear functional $\Lambda(f_1,f_2,f_3)$ satisfies  restricted weak-type estimates with exponents $\alpha$ if the following holds.

\underline{Case 1:  $\min(\alpha_1,\alpha_2,\alpha_3) \ge 0$.} Then we require existence of $j_0 \in \{1,2,3\}$ with the following property:  for every triple $(F_1,F_2,F_3)$ of finite Lebesgue measurable subsets of $\R$ we could find $B\subset F_{j_0}$ minor subset such that
\begin{eqnarray}\label{e.trilinear-alpha}\Lambda(f_1,f_2,f_3) &\lesssim& |F_1|^{\alpha_1} |F_2|^{\alpha_2} |F_3|^{\alpha_3}
\end{eqnarray}
for any $f_1,f_2,f_3$ with the following property:  $|f_j|\le 1_{F_j}$ if  $j\ne j_0$, $|f_{j_0}| \le 1_{F_{j_0}-B}$.

\underline{Case 2: $\min(\alpha_1,\alpha_2,\alpha_3)<0$.} Let $k$ be such that $\alpha_k<0$. By assumptions on $\alpha$ the other $\alpha_j$'s are nonnegative. Then we require the above property with $j_0=k$.

Let $A$ be the hexagon on the plane $L=\{\alpha_1+\alpha_2+\alpha_3=1\}$ with vertices
\begin{eqnarray}\label{e.A-def} 
A_1(-\frac 1 2, \frac 1 2, 1) , & A_2(\frac 1 2, - \frac 1 2, 1), & A_3(\frac 1 2, 1, -\frac 1 2), \\
\nonumber A_4(-\frac 1 2, 1, \frac 1 2), & A_5(1,-\frac 1 2, \frac 1 2), & A_6(1,\frac 1 2, -\frac 1 2)
\end{eqnarray}
By the interpolation argument of \cite{mtt2002}, to show Theorem~\ref{t.long-jump} and Theorem~\ref{t.short-jump}   it suffices to prove that in any given neighborhood (in the plane $L$) of any vertex of $A$ we could find $\alpha$ such that $\Lambda_{long}(f_1,f_2,f_3)$ and $ \Lambda_{short}(f_1,f_2,f_3)$ satisfy restricted weak-type estimates with exponents $\alpha$. (Note that when $\alpha$ is near a vertex of $A$ it is automatic that at most one coordinate of $\alpha$ could be negative.)

It will be clear from our proof (of the restricted weak-type estimates for all involved trilinear forms) that the index $j_0$ and the exceptional set $B$ depend only on $\alpha$ and $F_1,F_2,F_3$.  Also, in the proof the choice of $\alpha$ (inside any small neighborhoods of  any given vertices of $H$) will not depend on the underlying trilinear form.

Therefore, \emph{a posteriori}, to show the restricted weak-type estimates for $\Lambda_{short}$ it suffices to obtain the same estimate for $\Lambda_{short,s}$ (with the same set of exponents), provided that the implicit constants are uniform over $s\in [1,2]$. This uniformity in turn is a consequent of the fact that  the implicit constants in the assumptions for $K_s$  are uniform over $s\in [1,2]$.

Similarly, in the proof for $\Lambda_{long}$ we'll decompose it into a weighted sum of simpler trilinear forms, and it suffices to obtain the restricted weak-type estimates for each of the new forms (with the same set of exponents) provided that the implicit constants are uniform.

\section{Terminology of tiles and trees}\label{s.terms}

In this section we recall some terminologies from \cite{lacey2000,dtt2008,demeter2007} that will be used in the proof. By a $\emph{grid}$ we mean a collection of intervals whose lengths are integral powers of $2$  such that if $I,I'$ are two intersecting elements   then  $I \subset I'$ or $I' \subset I$. In addition to the standard grid $\calG_0$ of dyadic intervals $2^i[m,m+1)$, we will use the grids
\begin{eqnarray*}
\calG_{\ell,t} &=& \left\{ \left[2^i(m + \frac{\ell}{5}),2^i(m + \frac{\ell}{5} + 1)\right) : i = t \ (mod\ 4), m \in \BBZ \right\}
\end{eqnarray*}
where $\ell$ and $t$ are integers, clearly $G_{\ell,t}$ depends only on $\ell\ (mod \ 5)$ and $t\ (mod \ 4)$.  We will also make use of the grids 
\begin{eqnarray*}
\calG_{1} &=& \left\{ \left[2^i(m + \frac{(-1)^{i}}{3}),2^i(m + \frac{(-1)^i}{3} + 1)\right) : m \in \BBZ \right\} \\
\calG_{2} &=& \left\{ \left[2^i(m - \frac{(-1)^i}{3}),2^i(m - \frac{(-1)^i}{3} + 1)\right) :  m \in \BBZ \right\}.
\end{eqnarray*}

It is clear that for every (not necessarily dyadic) interval $I$ there is a $d \in \{0,1,2\}$ and a $J \in \calG_d$ such that $I \subset J$ and $J \subset 3I$; we then say that $I$ is \emph{d-regular}.

	A \emph{tile} $p$ is  a rectangle $I_p \times \omega_p \subset \rea^2$ of area $1$ such that $I_p$ is dyadic. A \emph{tri-tile}  $P$ will consist of a quadruplet of intervals $(I_P, \omega_{P_1},\omega_{P_2},\omega_{P_3})$ where $I_P$ is dyadic and $|I_P| |\omega_{P_i}|=1$ for each $i$. Associated to the tri-tile $P$ are the three tiles $P_i = I_P \times \omega_{P_i}$  which justify the notation that is implicit in the previous sentence. 
	
For each quadruplet  of integers  $\nu = (j_1,j_2,e,i)$ such that $0\le j_1,j_2\le 4$ and $498\le |e| \le 4002$ and $0\le i < 4000$, consider the collection  of tri-tiles \begin{eqnarray}
\nonumber\P_{\nu} \quad =\quad
\{ \Big( &&\big[2^{-i'}m,2^{-i'}(m+1)\big),\\
\nonumber &&\big[2^{i'}(n + \frac{j_1}{5}),  2^{i'}(n + \frac{j_1}{5} + 1)\big),\\
\label{e.Pnu-def} &&\big[2^{i'}(n + e + \frac{j_2}{5}), 2^{i'}(n + e + \frac{j_2}{5} + 1)\big),\\
\nonumber &&\big[2^{i'}(2n + e + \frac{j_1 + j_2}{5}+1), 2^{i'}(2n + e + \frac{j_1 + j_2}{5} + 2)\big)    \Big)   \\
\nonumber &&: m,n, i' \in \BBZ, i' = i \ (mod\ 4000)
\}
\end{eqnarray}
Above, we clearly have  $\omega_{P_1} \in \calG_{j_1,i}$, $\omega_{P_2} \in \calG_{j_2,i}$, and $\omega_{P_3} \in \calG_{j_1 + j_2,i}$.

Fixing $\nu$ for the remainder of the section  (some definitions below depend on $\nu$), we now recall, from \cite{demeter2007} (cf. \cite{mtt2002}), some notions of order for tiles. 
\begin{definition}
For two tiles $p,p'$  we write
\begin{itemize}
\item $p' < p$ if $I_{p'} \subsetneq I_p$ and $3\omega_p \subsetneq 3\omega_{p'}$
\item $p' \leq p$ if $p' < p$ or $p'=p$
\item $p' \lesssim p$ if $I_{p'} \subset I_p$ and $\omega_p \subset 10|e|\omega_{p'}$
\item $p' \lesssim' p$ if $p' \lesssim p$ and $10 \omega_{p'} \cap 10 \omega_p = \emptyset$
\end{itemize}
\end{definition}
It is not hard to see that if $P,P' \in \P_{\nu}$ are  two tri-tiles with $P_i' < P_i$ for some $i \in \{1,2,3\}$ then $P_{j}' \lesssim'P_j$ for each $j \in \{1,2,3\} \setminus \{i\}.$


The ordering above gives rise to the concept of a tree, which we recall below:

\begin{definition}\label{d.tree}
Let $i \in \{1,2,3\}$. An i-overlapping tree is a collection of tri-tiles $T \subset \P_{\nu}$ together with a top tri-tile $P_T \in \P_{\nu}$ which satisfies 
\begin{eqnarray*}
P_i < (P_T)_i \text{\ for all\ }P \in T \setminus \{P_T\}.
\end{eqnarray*}
We say that $T$ is a tree if it is an $i$-overlapping tree for some $i \in \{1,2,3\}$. We say that $T$ is a tree with top if $P_T\in T$.

A tree $T$ is called $j$-lacunary if 
\begin{eqnarray*}
P_{j} \lesssim' (P_{T})_j \text{\ for\ all\ }\ P \in T \setminus \{P_T\}. 
\end{eqnarray*} 
\end{definition}

It follows that a tree is $j$-lacunary if and only if it is $i$-overlapping for some $i \in \{1,2,3\} \setminus \{j\}$, furthermore for each $P \in T$ we have $\sgn(c(\omega_{P_j}) - c(\omega_{(P_T)_j})) = \epsilon_{i,j}$ where we define $\epsilon_{i,j} = \sgn(e)$ for $(i,j) \in \{(1,2), (1,3), (3,2)\}$ and $\epsilon_{i,j} =-\sgn(e)$ for $(i,j) \in \{(2,1), (3,1), (2,3) \}$. We will abbreviate $I_T := I_{P_T}.$

\begin{definition}\label{d.strong-j-disjoint}
We will say that a collection of trees $\T$ is \emph{strongly $j$-disjoint} for some $j \in \{1,2,3\}$ if
\begin{enumerate}
\item Each $T \in \T$ is $j$-lacunary
\item If $T,T' \in \T$ and $T \neq T'$ then $T \cap T' = \emptyset$ \label{firststrongdisjoint}
\item If $T,T' \in \T$, $T \neq T'$, $P \in T$, $P' \in T'$, and $\omega_{P_j} \subsetneq \omega_{P'_j}$ then $I_{P'} \cap I_{T} = \emptyset$ \label{secondstrongdisjoint}
\item  If $T,T' \in \T$, $T \neq T'$, and $P' \in T'$ then $P'_j \not\leq (P_{T})_j$ \label{thirdstrongdisjoint}
\end{enumerate}
\end{definition}
Note that, due to our choice of order on tiles, the condition (\ref{thirdstrongdisjoint}) above is somewhat nonstandard (in comparison with, say, \cite{demeter2007}).  Also note that conditions \eqref{firststrongdisjoint} and \eqref{secondstrongdisjoint} imply that if $T,T' \in \T$, $T \neq T'$, $P \in T$, and $P' \in T'$ then $P_j \cap P'_j = \emptyset.$

\section{Discretization}
\label{s.discretization}
In this section, we  discuss discretization, i.e. wavelet representation, for $\Lambda_{long}$ and $\Lambda_{short,s}$. We'll largely follow \cite{demeter2007}. We'll discuss in details the process for $\Lambda_{long}$, the discretization for $\Lambda_{short,s}$ will be similar and discussed at the end of the section.

\subsection{Cancellation between dilates}

The point of conditions \eqref{e.K-condition1},  \eqref{e.K-condition2} is that they allow one to decompose $K$ into much simpler kernels:
\begin{lemma}\label{l.Kdecomp} If $K$ satisfies \eqref{e.K-condition1} and \eqref{e.K-condition2} then
\begin{eqnarray}\label{e.Kdecomp}
\widehat{K}(\xi) &=& \sum_{j\in \Z}^\infty c_j \widehat{K}_j(2^j \xi) \ \ , \ \ \forall \ \xi \ne 0 \ \ ,
\end{eqnarray}
where $\{c_j\}_j \in \ell^1(\BBZ)$ and each $\widehat K_j$ could be furthermore written as  the sum of dilates of a single generating function  $\widehat{K}_j(\cdot) = \sum_{\ell \geq 0} \widehat{\phi}_j(2^\ell\cdot)$,
where $supp(\widehat \phi_j) \subset \{500 \leq |\xi| \leq 4000 \}$, 
and it holds uniformly in $j$ that
\begin{eqnarray} \label{e.phi}
|\phi_j^{(n)}(x)| &\le& C_{m,n} (1 + |x|)^{-m} 
\end{eqnarray}
for every $m,n \geq 0$. If $K$ satisfies \eqref{e.K-condition1} and \eqref{e.K-condition2} up to some high order then \eqref{e.phi} holds for $m,n\le M$ with $M$ comparably large.
\end{lemma}
The utility of this approach lies in the following cancellation between dilates of $K_j$: for every integers $k_1\le k_2$ we have $\widehat{K}_j(2^{k_2}\xi)  - \widehat K_j(2^{k_1}\xi) = \sum_{k_1 \le \ell < k_2} \widehat{\phi_j}(2^{\ell}\xi)$, which turns out to be convenient for reducing $\Lambda_{long}$ to  wavelet operators. Namely, by pulling out the sum in $j$, we thus see that the consideration of $\Lambda_{long}$ reduces to considerations of $\Lambda_j$, defined by:
\begin{eqnarray*}
\Lambda_{j}(f_1,f_2,f_3) &:=& \<\sum_{n=1}^L \Big(\sum_{k_{n-1} \le \ell  < k_n} B_{\phi_j,\ell}[f_1,f_2]\Big) a_n, f_3\> \ \ , \\
B_{\phi_j,\ell}[f_1,f_2](x) &:=& \int f_1(x + y) f_2(x - y) 2^{-\ell}\phi_j(2^{-\ell}y)\ dy \ \ ,
\end{eqnarray*}
and $B_{\phi_j,\ell}$ could be  decomposed into a finite number of discrete wavelet operators at scale $\ell$; this decomposition will be discussed in Section~\ref{s.wavelet-rep}.

\proof[Proof of Lemma~\ref{l.Kdecomp}] Observing that  the given assumptions on $\widehat K$ implies the existence of $\widehat K(0+)$ and $\widehat K(0-)$. We consider two cases. We'll only consider the setting when \eqref{e.K-condition1} and \eqref{e.K-condition2} hold for all orders; the finite order case could be achieved by the same argument.

\underline{Case 1:} Suppose that $\widehat K(0+)=\widehat K(0-)=0$, then using the given assumptions on $K$ it follows that for every $n\ge 0$ it holds that
\begin{eqnarray*}
|\frac {d^n}{d\xi^n}\widehat K(\xi)|
&\lesssim& |\xi|^{-n}\min(|\xi|, 1/|\xi|)
\end{eqnarray*}
(the improvement is at $n=0$). Let  $\eta$ be a nonnegative $C^\infty$ bump function on $\{1000\le |\xi|\le 2000\}$ such that $\sum_{j} \eta(2^j \xi) = 1$ for every $\xi \ne 0$. Let
\begin{eqnarray*}
\widehat K_j(\xi) &=& 2^{|j|} \widehat K(2^{-j}\xi)\eta(\xi) \ \ ,
\end{eqnarray*}
which is supported in $\{1000 \le |\xi| \le 2000\}$, it is routine to check that \eqref{e.Kdecomp} holds with $c_j=2^{-|j|}$, and $\frac{d^n}{d\xi^n}\widehat K_j(\xi) \lesssim_n 1$ for all $n\ge 0$. Let $\widehat \phi_j(\xi)=\widehat K_j(\xi) - \widehat K_j(2\xi)$, then $\widehat \phi_j$ has the desired properties.

\underline{Case 2:} $(\widehat K(0+), \widehat K(0-)) \ne (0,0)$. Let $\varphi$ be such that $\widehat \varphi$ is supported on $[-2000,2000]$ and is in $C^\infty(\R - \{0\})$, and $\widehat \varphi(\xi)=\widehat K(0+)$ for $\xi\in [0,1000]$ and $\varphi(\xi)=\widehat K(0-)$ for $\xi \in [-1000,0)$. Then by writing $\widehat K=\widehat \phi + (\widehat K-\widehat \phi)$ and applying the  analysis in Case 1 for $K-\phi$, we are left with $\widehat \varphi$, for which we will decompose directly into the sum of dilates of a single generating function. Namely, let $\widehat \phi(\xi)=\widehat \varphi(\xi)-\widehat \varphi(2\xi)$, it is clear that $\widehat \phi$ satisfies the desired properties.
\endproof

\subsection{Wave packet representation}\label{s.wavelet-rep} 
Below we  will decompose $B_{\phi_j,\ell}[f_1,f_2]$ into wavelet sums. For convenience of notation, we will suppress the variable $j$, namely below $\phi=\phi_j$ whose Fourier transform is supported on $\{500\le |\xi|\le 2000\}$ and $\phi$ satisfies  \eqref{e.phi} up to sufficiently high order.

\begin{definition}\label{d.wavepacket} We say that $\psi$ is an $L^2$ wave packet of order $M$ adapted to a tile $p=I\times \omega$ if $\widehat \psi$ is supported in $\omega$ and  the following estimate holds for all $0\le m,n\le M$:
\begin{eqnarray}  \label{e.wavepacketdecay}
\frac{d^n}{dx^n}\Big(\e^{- 2 \pi \i c(\omega) x} \psi(x)\Big)   &\le&  C_{M,N,n,m} \frac{1}{|I|^{n+\frac 1 2}}  \widetilde \chi_{I}(x)^m
\end{eqnarray}
\end{definition}

\begin{lemma}\label{l.wavelet-rep} Given any $M,N>0$, if  $\phi$ satisfies  \eqref{e.phi} up to sufficiently high order and $\widehat \phi$ is supported in $\{500\le |\xi|\le 4000\}$ then for each $\ell\in \Z$
$B_{\phi,\ell}[f_1,f_2](x)$ can be written as the sum over   $\nu=(j_1,j_2,e)$, $0\le j_1,j_2\le 4$ and $498\le |e| \le 4002$ integers, of
\begin{eqnarray} \label{e.Bphiell-decomp}
\sum_{j\in \Z} 2^{-N|j|} \sum_{P \in \P_\nu: |\omega| = 2^{-\ell}} |I_P|^{-1/2}\<f_1, \psi_{j,P,1}\>\<f_2, \psi_{j,P,2}\> \psi_{j,P,3}(x) \ \ ,
\end{eqnarray}
where (uniform over tri-tiles $P\in \P_\nu$, $i=1,2,3$, and $j\ge 0$) $\psi_{j,P,i}$ is an $L^2$ wave packet adapted to $P_i$ of order $M$ (the  constants in \eqref{e.wavepacketdecay} may depend on $M,N,\nu$). 
\end{lemma}
\proof We consider $M=\infty$ below, the finite case is similar. Recall that the Fourier transform $\mathcal F$ is defined by \eqref{e.fourier-transform}. We first make several remarks about $B_{\phi,\ell}$. Suppose that the supports of $\widehat \psi_1$ and $\widehat \psi_2$ are contained in intervals $\omega_1,\omega_2$  respectively. Then the identity
\begin{eqnarray*} 
\mathcal F\left[B_{\phi,\ell}[\psi_1,\psi_2]\right](\xi) &=&  \int_{\R} \widehat{\psi}_1(\xi - \eta) \widehat{\psi}_2( \eta) \widehat{\phi}(2^l(2 \eta - \xi))\ d\eta 
\end{eqnarray*}
 gives rise to two observations. First, if $B_{\phi,\ell}(\psi_1,\psi_2)$ does not vanish then 
\begin{eqnarray} \label{diffsupport}
(\omega_1 - \omega_2) \cap \{500\cdot 2^{-\ell} &\le& |\xi| \leq 4000 \cdot 2^{-\ell}\} \neq \emptyset.
\end{eqnarray}
Second, the support of Fourier transform of $B_{\phi,\ell}(\psi_1,\psi_2)$   is contained in 
\begin{eqnarray} \label{sumsupport}
\omega_1 + \omega_2.
\end{eqnarray}
Now, turning to spatial localization,  we have
\begin{eqnarray} \label{modulationofbilinear}
\e^{- 2 \pi \i 2 \xi x} B_{\phi,\ell}[\psi_1,\psi_2](x)  &=&  B_{\phi,\ell}[ \e^{- 2 \pi \i \xi (\cdot)}\psi_1, \e^{- 2 \pi \i \xi (\cdot)}\psi_2](x) \ \ .
\end{eqnarray}
Since $\phi$ satisfies \eqref{e.phi}, it follows that if for some $\xi$, $x_1$, $x_2$ we have
\begin{eqnarray*}
\frac{d^n}{d x^n}\left(\e^{- 2\pi \i \xi x} \psi_i(x) \right) &\le& C_{n,m} 2^{-\ell (n+\frac 1 2)} (1 + \frac{|x - x_i|}{2^\ell})^{-m}
\end{eqnarray*}
for each $i=1,2$, and for every $n,m \geq 0$,  then
\begin{eqnarray} \label{derivativeofbilinear}
&&\frac{d^n}{dx^n} \Big(\e^{- 2 \pi \i 2 \xi x} B_{\phi,\ell}[\psi_1,\psi_2](x) \Big)  \quad \le \\
\nonumber &\le& C_{n,m} 2^{-\ell (n+1)} (1 + \frac{|x_1 - x_2|}{2^\ell})^{-m} (1+\frac{|x - (x_1 + x_2)/2|}{2^\ell})^{-m}
\end{eqnarray}
for each $n,m \geq 0$. (Here we emphasize that $C_{m,n}$'s are independent of $\xi$.) 

Fix a Schwartz function $\widehat{\psi}$ supported on $[0,2/5)$ such that $\sum_{j \in \BBZ} |\widehat{\psi}(\cdot - \frac{j}{5})|^2 = 1$. For each pair of intervals $(I,\omega)$ with $|\omega||I|=1$ let 
\begin{eqnarray*}
\widehat{\psi}_{I,\omega} = \e^{- 2 \pi \i c(I) \xi} |\omega|^{-1/2} \widehat{\psi}(|\omega|^{-1}(\xi - c(\omega))) 
\end{eqnarray*}
which is supported inside  the right half of $\omega$. Using a Fourier sampling theorem,  for  any Schwartz function $f$ it holds that
\begin{eqnarray*}
f = \sum_{j = 0}^{4} \sum_{\substack{\omega \in \calG_0 \\ |\omega| = 2^{-\ell}}} \sum_{\substack{I \in \calG_0 \\ |I| = 2^\ell}} \<f,\psi_{I,\omega + \frac{j}{5}|\omega|}\> \psi_{I,\omega + \frac{j}{5}|\omega|}\ \ \ \ .
\end{eqnarray*}
Let $\omega_1 = \omega + \frac{j_1}{5}|\omega|$ and $\omega_2= \omega + (\frac{j_2}{5} + e) |\omega|$. By \eqref{diffsupport}, it follows that $B_{\phi,\ell}[f_1,f_2](x)$ can be written as the sum over triplets of integers $(j_1,j_2,e)$, with $0\le j_1,j_2\le 4$ and $498\le |e| \le 4002$, of
\begin{eqnarray} \label{sumofwavepackets}
\sum_{\substack{\omega \in \calG_0 \\ |\omega| = 2^{-\ell}}} \sum_{\substack{I_1 \in \calG_0 \\ |I_1| = 2^{\ell}}} \sum_{\substack{I_2 \in \calG_0 \\ |I_2| = 2^{\ell}}} \<f_1, \psi_{I_1, \omega_1}\>\<f_2, \psi_{I_2, \omega_2}\> \varphi_{I_1,I_2,\omega,j_1,j_2,e}(x) \ \ ,
\end{eqnarray}
\begin{eqnarray*}
\varphi_{I_1,I_2,\omega,j_1,j_2,e} (x) &:=& 
B_{\phi,\ell}[\psi_{I_1,\omega + \frac{j_1}{5}|\omega|}, \psi_{I_2,\omega + (\frac{j_2}{5} + e) |\omega|}](x) \ \ \ .
\end{eqnarray*}
By \eqref{sumsupport}, $\widehat{\varphi}_{I_1,I_2,\omega,j_1,j_2,e}$ is supported on $\omega_3:=[0,|\omega|) + 2 c(\omega) + (\frac{j_1 + j_2}{5} + e)|\omega|$
and, by \eqref{derivativeofbilinear}, satisfies
\begin{eqnarray}  \label{phidecay}
&&\frac{d^n}{dx^n}\Big( \e^{- 2 \pi \i (2c(\omega) + (\frac{j_1 + j_2}{5} + e)|\omega|) x} \varphi_{I_1,I_2,\omega,j_1,j_2,e}(x)\Big) \leq   \\
\nonumber &\le&  C_{n,m} 2^{-\ell(n+1)}  (1 + \frac{|c(I_1) - c(I_2)|}{2^\ell})^{-m} (1+\frac{|x - (c(I_1) + c(I_2))/2|}{2^\ell})^{-m}\\
\nonumber &\le&  C_{n,m} 2^{-\ell(n+1)}  (1 + \frac{|c(I_1) - c(I_2)|}{2^\ell})^{-m} (1+\frac{|x - c(I_1)|}{2^\ell})^{-m}
\end{eqnarray}
for each $n,m \geq 0$. (Note that $j_1,j_2,e$ are bounded.)

We now fix $I_1$ and further divide the right hand side of \eqref{sumofwavepackets} according to  $j:=2^{-\ell}(c(I_1)-c(I_2))$. For $j=0$ i.e. for terms in the sum \eqref{sumofwavepackets} where $I_1 = I_2=:I$ we may define the tri-tile $P=(I,\omega_1,\omega_2,\omega_3)$ and the corresponding wave packets naturally 
\begin{eqnarray*}
\psi_{0,P,1} &=& \psi_{I_1,\omega_1} \ \ , \ \ \psi_{0,P,2} \quad =\quad \psi_{I_2,\omega_2} \ \ , \ \ \psi_{0,P,3}\quad=\quad |I|^{1/2} \psi_{I_1,I_2,\omega,j_1,j_2,e} \  \ .
\end{eqnarray*}

The remaining terms can be dealt with by using the rapid decay in $|c(I_1) - c(I_2)|$,  \eqref{phidecay}: we still define $\psi_{j,P,1}=\psi_{I_1,\omega_1}$, however to shift the localization of $\psi_{I_2, \omega_2}$ to $I_1$ we define
\begin{eqnarray*}
\psi_{j,P,2} &:=& (1 + \frac{|c(I_1) - c(I_2)|}{2^\ell})^{-L} \cdot \psi_{I_2, \omega_2} \ \ ,\\
\psi_{j,P,3} &:=& 2^{Nj}  (1 + \frac{|c(I_1) - c(I_2)|}{2^\ell})^{L} \big(|I|^{1/2} \psi_{I_1,I_2,\omega,j_1,j_2,e})
\end{eqnarray*}
for some large $L$. (The rapid decay in \eqref{phidecay} takes care of the extra factors in $\psi_{j,P,3}$.)

Finally, we split \eqref{sumofwavepackets} up one more time so that  whenever $|I| >|I'|$ we have $|I| = 2^{4000k} |I'|$ for some positive integer $k$.  Note that while this splitting gives rise to the sparseness required by $\P_\nu$, it also means that we need to relabel and rescale the $\phi_{j,P,i}$ slightly to maintain the sequence of weights $2^{-N|j|}$. \endproof

It follows that to prove restricted weak-type estimates for $\Lambda_{long}$, we are left with showing the following theorem. In the theorem, $\nu=(j_1,j_2,e,i)$ is any quadruplet of integers such that $0\le j_1,j_2\le 4$, $498\le |e|\le 4002$, $0\le i<4000$, and $\P_\nu$ is defined by \eqref{e.Pnu-def}.

\begin{theorem} \label{t.model-varnorm}
Let $r > 2$ and $p_1, p_2,q$ satisfy \eqref{e.BHTrange}.

Suppose that (uniformly over tri-tiles $P \in \P_{\nu}$, $i=1,2,3$),  $\psi_{P,i}$ is an $L^2$ normalized wave packet adapted to $P_i$ up to order $M$ sufficiently large (the required $M$ may depend on $p_1,p_2,q$).

Then the trilinear form
\begin{eqnarray*}
\<\sum_{n=1}^L \sum_{
P \in \P_{\nu}: \ 2^{k_{n-1}}\le |I_P| < 2^{k_n}} |I_{P}|^{-1/2} \<f_1,\psi_{P,1}\> \<f_2,\psi_{P,2}\> \psi_{P,3}\ a_n, f_3\>
\end{eqnarray*}
satisfies restricted weak-type estimates with  exponents $\alpha$ arbitrarily close to any given vertex of $A$ defined by \eqref{e.A-def}.
\end{theorem}
 
Recall that $\sum_{n=1}^{L(x)} |a_n(x)|^{r'} \le 1$. For convenience, let $a_{P}(x) := a_{m}(x)$ if $m = m(P,x)$ is the unique integer in $\{1,\dots, L(x)\}$ satisfying $2^{k_{m-1}(x)} \leq |I_P| < 2^{k_{m}(x)}$, and $a_P(x)=0$ if such $m$ does not exist. We also  let $\phi_{P,3}(x) = a_p(x)\psi_{P,3}(x)$ and $\phi_{P,i} = \psi_{P,i}$ for $i=1,2$. Let $\P$ be a finite subset of $\P_\nu$.  It suffices to demonstrate that the trilinear form
\begin{eqnarray} \label{maintheoremlinbound}
 \Lambda_{\P}(f_1,f_2,f_3) = \<\sum_{P \in \P} |I_{P}|^{-1/2}\<f_1,\phi_{P,1}\> \<f_2,\phi_{P,2}\> \phi_{P,3},f_3\>
\end{eqnarray}
is of restricted weak type with exponents $\alpha$, with  $\P$-uniform implicit constants.

\subsection{Discretization for $\Lambda_{short,s}$}

Recall that
\begin{eqnarray*}
\Lambda_{short,s}(f_1,f_2,f_3) &=&   \<\sum_{n \in \Z} \int f_1(\cdot +y) f_2(\cdot - y) 2^{-n} K_s(2^{-n}y)dy d_n(s,\cdot), f_3\>
\end{eqnarray*}
Using $\widehat K_s(\xi) \lesssim |\xi|$ we could proceed as in Case 1 of the proof of Lemma~\ref{l.Kdecomp} and obtain a decomposition
\begin{eqnarray*}
\widehat K_s(\xi) &=& \sum_{j\in \Z} 2^{-|j|} \widehat K_{s,j}(2^j \xi)
\end{eqnarray*}
where $\widehat K_{s,j}$ is supported in $\{1000\le |\xi| \le 2000\}$ and $\frac{d^n}{\xi^n} \widehat K_{s,j}(\xi) \lesssim_n 1$ for all $n\ge 0$. Thus it suffices to consider restricted weak-type estimates for
\begin{eqnarray*} &&\<\sum_{n\in \Z} \int f_1(\cdot +y) f_2(\cdot - y) 2^{-(n+j)} K_{s,j}(2^{-(n+j)}y)dy \  d_n(s,\cdot), f_3 \>\\
&=& \<\sum_{n \in \Z} B_{K_{s,j},n}[f_1,f_2](\cdot) d_n(s,\cdot), f_3\> \  \ .
\end{eqnarray*}
Now using Lemma~\ref{l.wavelet-rep} with $K_{s,j}$ playing the role of $\phi$, it follows that to obtain the desired restricted weak-type estimates  for $\Lambda_{short}$ we are left with showing the following theorem.  Below, $\nu=(j_1,j_2,e,i)$ is any quadruplet of integers such that $0\le j_1,j_2\le 4$, $498\le |e|\le 4002$, $0\le i<4000$, and $\P_\nu$ is defined by \eqref{e.Pnu-def}.

\begin{theorem} \label{t.model-square}
Let $r > 2$ and $p_1, p_2,q$ satisfy \eqref{e.BHTrange}.

Suppose that (uniformly over tri-tiles $P \in \P_{\nu}$, $i=1,2,3$),  $\psi_{P,i}$ is an $L^2$ wave packet adapted to $P_i$ up to order $M$ sufficiently large.

Let $(d_n)_{n\in \Z}$ be a sequence of measurable functions such that $\sum_n |d_n(x)|^2 \le 1$.

Then the trilinear form
\begin{eqnarray*}
\<\sum_{n \in \Z} \sum_{P \in \P_{\nu}: \  |I_P|=2^n} |I_{P}|^{-1/2} \<f_1,\psi_{P,1}\> \<f_2,\psi_{P,2}\> \psi_{P,3}\ d_n, f_3\>
\end{eqnarray*}
satisfies restricted weak-type estimates with exponents $\alpha$ arbitrarily close to any vertex of $A$ defined by \eqref{e.A-def}.
\end{theorem}

\section{Auxiliary estimates}\label{s.auxiliary}

The following bound follows from the L\'{e}pingle inequality and a square function argument, see \cite{jsw2008} for details.

\begin{lemma} \label{l.lepingle}
Let $\zeta$ be any Schwartz function. Let $\zeta_k(\cdot)=2^{-k}\zeta(2^{-k}\cdot)$. Then for $r > 2$ and $1 < s < \infty$ it holds that
\begin{eqnarray*}
\|(\zeta_k*f)(x)\|_{L_x^s(V^r_k(\Z))} \lesssim_{r,s} \|f\|_{L^{s}}
\end{eqnarray*} 
\end{lemma}

Next, we have a Rademacher-Menshov type lemma:

\begin{lemma} \label{vnrmlemma}
Let  $f_1, \ldots, f_N$ be functions on a measure space $X$ such that for every sequence of signs $\epsilon_1, \ldots, \epsilon_N \in \{1,-1\}$ it holds that
\begin{eqnarray} \label{randomsignassumption}
\|\epsilon_1 f_1 + \epsilon_2 f_2 +\dots + \epsilon_N f_N\|_{L^2} &\le& B.
\end{eqnarray}
Then
\begin{eqnarray*}
\|\sum_{j = 1}^n f_j(x)\|_{L^2_x(V^2_n)} \lesssim (1 + \log(N)) B.
\end{eqnarray*}
\end{lemma}

\begin{proof}
We rewrite  $\sum_{j = 1}^n f_j(x) = f_n(x) + \sum_{1\le j < n} f_j(x)$. Estimating the $V^2$ norm by the $\ell^2$ norm, it is clear that
\begin{eqnarray*}
\|f_n(x)\|_{L^2_x(V^2_n)} &\lesssim& \left( \sum_{j = 1}^{N}\|f_j\|^2_{L^2}\right)^{1/2} \quad \lesssim\quad \quad \E_{\epsilon_1,\dots,\epsilon_n} \|\sum_j \epsilon_j f_j\|_{L^2}  \quad \lesssim \quad B \ \ .
\end{eqnarray*}
  
It remains to consider the contribution from $\sum_{1\le j < n}f_j$. For each $n \in \{1,\dots, N\}$ we will decompose $[1,n)$ into  disjoint subintervals,
\begin{eqnarray*}
[0,n) = \bigcup_{m \leq \log_2(N)} \omega_{n,m} \ \ ,
\end{eqnarray*}
as follows: Let $I$ be the dyadic interval of length $2^{m+1}$ that contains $n$. If $n$ is on the left half of $I$ then let $\omega_{n,m} = \emptyset$. 
If $n$ is on the right half of $I$ then let $\omega_{n,m}$ be the left half of $I$. It follows that
\begin{eqnarray*}
\|\sum_{j < n} f_j(x)\|_{L^2_x(V^2_n)} &\le& \sum_{0\le m \le \log_2(N)} \|\sum_{j \in \omega_{n,m}} f_j(x)\|_{L^2_x(V^2_n)}.
\end{eqnarray*}
Since for each $m$, $\omega_{n,m}$ is constant (in $n$) on dyadic intervals of length $2^m,$ we have 
\begin{eqnarray*}
\|\sum_{j \in \omega_{n,m}} f_j(x)\|_{L^2_x(V^2_n)} &\lesssim& \|(\sum_{\omega \ dyadic: \ |\omega| = 2^m}|\sum_{j \in \omega} f_j(x)|^2)^{1/2}\|_{L^2_x} \\ &\le& B. 
\end{eqnarray*}
here the final inequality follows from another appeal to \eqref{randomsignassumption} using sequences $(\epsilon_j)$ that are constant (as functions of $j$) on dyadic intervals of length $2^m$.
\end{proof}

We'll also use a Bessel inequality, Lemma~\ref{nonmaxbessellemma}. For a proof see e.g.  \cite[Proposition 13.1]{dtt2008}. Below recall that $\psi_{P,j}$ are (unmodified) $L^2$-normalized Fourier wave packet up to sufficiently high orders:

\begin{lemma} \label{nonmaxbessellemma}
Let $j\in \{1,2,3\}$. Let $\T$ be a collection of strongly $j$-disjoint trees, let $\Q = \bigcup_{T \in \T}T$, and suppose that $\|\sum_{T \in \T}1_{I_T}\|_{L^{\infty}} \leq L.$ Then for any sequence of coefficients $\{b_P\}_{P \in \Q}$
\begin{eqnarray*}
\|\sum_{P \in \Q} b_P \psi_{P,j}\|_{L^2} &\lesssim& \log(1 + L)\|b_P\|_{\ell^2(\Q)}.
\end{eqnarray*}
\end{lemma}

\section{A variation-norm multiplier estimate}\label{s.varnorm-multiplier}

In this section we consider a variation-norm version of Bourgain \cite[Lemma 4.11]{bourgain1990}, namely Theorem~\ref{t.bourgain-varnorm} below. In the following, let $\xi_1 < \ldots < \xi_N$ be real numbers. For each integer $k$ we denote the sharp multi-frequency projection at scale $k$ onto $\xi_1,\dots, \xi_N$  by  $\Pi_k[f] = \mathcal F^{-1}[1_{R_k}\widehat{f}]$, where  $R_k = \bigcup_{\xi \in \Omega} (\xi - 2^{-k}, \xi +2^{-k})$.
\begin{theorem} \label{t.bourgain-varnorm}
For every $r>2$ and  $\epsilon > 0$ it holds that
\begin{eqnarray*}
\|\Pi_k[f](x)\|_{L^2_x(V^r_k)} &\le& C_{r,\epsilon} N^{\epsilon} \|f\|_{L^2} \ \ .
\end{eqnarray*}
\end{theorem}

 A variant of Theorem~\ref{t.bourgain-varnorm} with  smooth multi-frequency projections was considered in \cite{not2010}, where a range of $L^p$ estimates was obtained; for the current paper we need sharp  frequency projections, but $L^2$ is sufficient.

The starting point of our proof is Lemma 3.2 from \cite{not2010}:
\begin{lemma}  \label{l.exponentialsum}
Suppose that  $\{c_k\}_{k=0}^{\infty}$ is a sequence in $\rea^N$, and $2 < q < r$. Then
\begin{eqnarray*}
\|   \sum_{j=1}^N c_{k,j} {\bf e} ^{2 \pi {\bf i} \xi_j y} \|_{L^2_{y\in[0,1]}(V^r_{k\ge 0})}  &\le& C N^{(\frac{1}{2} - \frac{1}{q})\frac{r}{r-2}}\|c_k\|_{V^q_{k\ge 0}(\ell^2(\rea^N))}
\end{eqnarray*}
where $C$ may depend on $r,q$ and $\min_j |\xi_j - \xi_{j-1}|.$
\end{lemma}

Through a standard averaging argument (see e.g. the proof of Proposition 4.1 in \cite{not2010}), the lemma above gives

\begin{proposition} \label{p.bourgain-varnorm-sep}
Let $\chi$ be a smooth function such that $\widehat \chi$ is identically one on $[-0.9,0.9]$ and supported on $[-1,1]$. Assume that  $\xi_{j+1} \ge \xi_j + 1$ for each $j$, and let $\chi_{k,j}$ be defined by $\widehat \chi_{k,j}(\xi)=\chi(2^k(\xi-\xi_j))$. Then for $r>2$ and $\epsilon>0$ it holds that
\begin{eqnarray*}
\| \sum_{1\le j \le N} (\chi_{k,j}\ast f)(x)\|_{L^2_x(V^r_{k \geq 0})} &\le& C_{r,\epsilon,\chi} N^{\epsilon}\|f\|_{L^2}
\end{eqnarray*}
\end{proposition}
\proof To keep the paper self-contained, we sketch the averaging argument. Let $f_j(x) = \mathcal F^{-1}[1_{|\xi-\xi_j|\le 1} \widehat f(\xi)](x)$. Since $\xi_j$'s are separated, we have $\|f\|_2 \approx \|(f_j)\|_{L^2(\ell^2_j)}$. Let $M$ be the best constant such that if $supp(\widehat g_j)\subset [\xi_j-1, \xi_j+1]$ for all $j$ then
\begin{eqnarray*}
\|\sum_{1\le j \le N}  (\chi_{k,j}  \ast g_j)(x)\|_{L^2_x(V^r_{k\ge 0})} &\le& M \|(g_j)\|_{L^2_x(\ell^2_j)} \ \ ,
\end{eqnarray*}
by the triangle inequality and Lemma~\ref{vnrmlemma} it is clear that $M=O_{\chi,r}(N)<\infty$. Our aim is to show that $M=O_{\epsilon,r,\chi}(N^\epsilon)$. Since $\widehat g_j$ is supported on $[\xi_j-1,\xi_j+1]$, for $|y|$ small we have
\begin{eqnarray*}
\| g_j(x)  - \e^{2\pi \i \xi_j y} g_j(x-y)\|_{L^2_x(\ell^2_j)} &\lesssim& |y| \|g_j\|_{L^2_x(\ell^2_j)} \ \ .
\end{eqnarray*}
Averaging over $0\le y\le \delta$ with $1 \lesssim \delta<1$ sufficiently small, we obtain
\begin{eqnarray*}
&&\|\sum_{1\le j \le N} (\chi_{k,j}\ast g_j)(x)\|_{L^2_x(V^r_{k\ge 0})} \\
&\le& C_\delta \|\sum_{1\le j \le N} \e^{2\pi \i \xi_j y} (\chi_{k,j}\ast g_j)(x-y)\|_{L^2_{y\in [0,1]}(L^2_x(V^r_{k\ge 0}))} + \frac {M}2 \|g_j\|_{L^2_x(\ell^2_j)}\\
&=& C_\delta \|\sum_{1\le j \le N} \e^{2\pi \i \xi_j y} (\chi_{k,j}\ast g_j)(x)\|_{L^2_x(L^2_{y\in [0,1]}(V^r_{k\ge 0}))} + \frac {M}2 \|g_j\|_{L^2_x(\ell^2_j)}
\end{eqnarray*}
using translation invariant and Fubini. Using Lemma~\ref{l.exponentialsum} for each fixed $x$ it follows that for $q >2$ sufficiently close to $2$ (depending on $\epsilon>0$) we have
\begin{eqnarray*}
\|\sum_{1\le j \le N} (\chi_{k,j}\ast g_j)(x)\|_{L^2_x(V^r_{k\ge 0})} 
&\le& C_{\epsilon, q, r} N^\epsilon \| (\chi_{k,j}\ast g_j)(x)\|_{L^2_x(V^q_k(\ell^2_j))} + \frac {M}2 \|g_j\|_{L^2_x(\ell^2_j)}\\
&\le& C_{\epsilon, q, r} N^\epsilon  \| (\chi_{k,j}\ast g_j)(x)\|_{\ell^2_j(L^2_x(V^q_k))} + \frac {M}2 \|g_j\|_{L^2_x(\ell^2_j)}\\
&\le& C_{\epsilon, q, r,\chi} N^\epsilon \|g_j\|_{\ell^2_j(L^2_x)} + \frac{M}2 \|g_j\|_{L^2_x(\ell^2_j)} \ \ ,
\end{eqnarray*}
in the second estimate we used $q>2$ and in the last estimate we used Lemma~\ref{l.lepingle}. Since this holds for arbitrary $(g_j)$ satisfying $supp(\widehat g_j)\subset [\xi_j-1,\xi_j+1]$, by definition of $M$ we obtain $M \le C_{\chi,\epsilon, q, r} N^\epsilon  + \frac M 2$, therefore $M = O_{\epsilon,r,\chi}(N^\epsilon)$ as desired. \endproof

Now, using Proposition \ref{p.bourgain-varnorm-sep} and  a simple square function argument, we obtain a frequency separated version of Theorem~\ref{t.bourgain-varnorm}. Namely, if  $\xi_{j+1} \geq \xi_j + 1$ for each $j$ then for $r>2$ and $\epsilon>0$ it holds that
\begin{eqnarray*}
\| \Pi_k[f]\|_{L^2_x(V^r_{k \geq 0})} &\le& C_{r,\epsilon} N^{\epsilon} \|f\|_{L^2} \ \ .
\end{eqnarray*}

To remove the frequency separation requirement $\xi_{j+1} \geq \xi_{j} + 1$, we will need the following estimate, which will be proved using  Lemma \ref{vnrmlemma}. 

\begin{proposition} \label{bourgainrmprop}
Suppose that $S$ is a finite set of integers. Then
\begin{eqnarray} \label{bourgainrm}
\| \Pi_k[f]\|_{L^2_x(V^2_{k \in S})} &\le& C (1 + \log(|S|)) \|f\|_{L^2}
\end{eqnarray}
\end{proposition}

\begin{proof}
Let $n=|S|$. Let $s_1< \ldots< s_n$ be elements of $S$. For $j =1, \ldots, n - 1$ write
\begin{eqnarray*}
f_j = \mathcal F^{-1}[(1_{R_{s_j}} - 1_{R_{s_{j+1}}})\widehat{f}]
\end{eqnarray*}
and let $f_n = \mathcal F^{-1}[1_{R_{s_n}}\widehat{f}]$. Then, the $f_j$ are orthogonal in $L^2(\R)$, and $\Pi_{s_k}[f] = \sum_{j \geq k} f_{j}$, 
so Lemma \ref{vnrmlemma} gives \eqref{bourgainrm}.
\end{proof}

\begin{proof} [Proof of Theorem \ref{t.bourgain-varnorm}]
By monotone convergence, it suffices to prove
\begin{eqnarray*}
\|\Pi_k[f](x)\|_{L^2_x(V^r_{k \in [a,b]})} &\le& C_{r,\epsilon} N^{\epsilon} \|f\|_{L^2}
\end{eqnarray*}
for every  finite interval $[a,b]$, provided that the constant is independent of $[a,b].$ 

Now, we may choose $\{k_j\}_{j=0}^N$ with 
\begin{eqnarray*}
a = k_0 \le \ldots \le k_N = b
\end{eqnarray*}
so that the number of connected components of $R_k$ is constant on each interval $[k_j, k_{j+1}).$
Then, for each $x$
\begin{eqnarray*}
\|\Pi_k[f](x)\|_{V^r_{k \in [a,b]}} &\le& C \|\Pi_{k_j}[f](x)\|_{V^r_{j \in [0,N]}} + C\left( \sum_{j = 1}^{N}\|\Pi_k[f](x)\|^2_{V^r_{k \in [k_{j-1},k_j)}} \right)^{1/2}.
\end{eqnarray*}
The contribution to the $L^2_x$ norm of the first term on the right above is acceptable by Proposition \ref{bourgainrmprop}. Furthermore, for each $k,k' \in [k_{j-1}, k_j)$ we have
\begin{eqnarray*}
\Pi_{k}[f](x) - \Pi_{k'}[f](x) = \Pi_{k}[f_j](x) - \Pi_{k'}[f_j](x)
\end{eqnarray*}
where $f_j = \mathcal F^{-1}[(1_{R_{k_{j-1}}} - 1_{R_{k_j}})\widehat{f}]$. Thus, using the orthogonality of the $f_j$ it suffices to show that, for each $1\le j\le N$, 
\begin{eqnarray*}
\|\Pi_k[f](x)\|_{L^2_x(V^r_{k \in [k_{j-1,k_j})})} &\le& C_{r,\epsilon} N^{\epsilon} \|f\|_{L^2}  \ \ .
\end{eqnarray*}

Fix $j$ and let $M$ be the (constant) number of connected components of $R_k$ for $k \in [k_{j-1},k_j)$, clearly $M\le N$. For $k \in [k_{j-1},k_j)$ we can write $R_k$ as the disjoint union of open intervals 
\begin{eqnarray*}
R_k = \bigcup_{1\le \ell \le M}I_{\ell,k}
\end{eqnarray*}
where $I_{\ell,k'} \subset I_{\ell,k}$ for $k' > k.$
Let $\Omega=\{\xi_1, \ldots, \xi_N \}$ and define  
\begin{eqnarray*}
L^{\sharp} &:=& \{\ell \in [1,M] : |I_{\ell,k_{j-1}} \cap \Omega| = 1\} \quad , \quad \Pi^{\sharp}_k[f] \quad :=\quad \sum_{\ell \in L^{\sharp}}\mathcal F^{-1}[1_{I_{\ell,k}}\widehat{f}] 
\end{eqnarray*}
and let $L^{\flat} = [1,M] - L^{\sharp}$ and define $\Pi^{\flat}$ analogously. Clearly $\Pi_k = \Pi^{\sharp}_k + \Pi^{\flat}_{k}$. Rescaling by a factor of $2^{k_{j-1}}$, an application of the known frequency-separated case immediately gives
\begin{eqnarray*}
\|\Pi^{\sharp}_k[f](x)\|_{L^2_x(V^r_{k \in [k_{j-1,k_j})})} &\le& C_{r,\epsilon} |L^{\sharp}|^{\epsilon} \|f\|_{L^2}
\end{eqnarray*} 
and so it remains to consider $\Pi^{\flat}.$ For each $\ell \in L^{\flat}$ let 
\begin{eqnarray*}
\zeta_{\ell} &=& \min(I_{\ell,k_{j-1}} \cap \Omega) \ \ , \ \ \rho_{\ell} \quad = \quad \max(I_{\ell,k_{j-1}} \cap \Omega)
\end{eqnarray*}
and $I'_{\ell} = (\zeta_\ell,\rho_\ell)$, so that 
$I_{\ell,k}= (\zeta_\ell - 2^{-k}, \zeta_\ell] \cup I'_\ell \cup [\rho_\ell, \rho_\ell + 2^{-k})$. Now, define
\begin{eqnarray*}
\Pi^{\flat,1}_k[f] &=& \sum_{\ell \in L^{\flat}} \mathcal F^{-1}[1_{(\zeta_\ell - 2^{-k}, \zeta_\ell + 2^{-k})}\widehat{f}] \\ 
\Pi^{\flat,2}_k[f] &=& \sum_{\ell \in L^{\flat}} \mathcal F^{-1}[1_{(\rho_\ell - 2^{-k}, \rho_\ell + 2^{-k})}\widehat{f}] \ \ .
\end{eqnarray*}
For $k\in [k_{j-1},k_j)$ we obtain the decomposition
\begin{eqnarray*}\Pi^{\flat}_k[f] &=& g + \Pi^{\flat,1}_k[h_1] + \Pi^{\flat,2}_k[h_2]
\end{eqnarray*}
where $g = \mathcal F^{-1} [\sum_{\ell \in L^{\flat}}I'_\ell \widehat{f}]$ (which stays the same under $\Pi^{\flat}$) and
\begin{eqnarray*}
h_1 \quad = \quad \mathcal F^{-1}[\sum_{\ell \in L^{\flat}}1_{(\zeta_\ell - 2^{-k_{j-1}}, \zeta_\ell]} \widehat{f}]  & , & h_2 \quad=\quad \mathcal F^{-1} [\sum_{\ell \in L^{\flat}}1_{[\rho_\ell,\rho_\ell + 2^{-k_{j-1}})} \widehat{f}]
\end{eqnarray*}

Rescaling by a factor of $2^{k_{j-1}}$, an application of the known frequency-separated case then gives
\begin{eqnarray*}
\|\Pi^{\flat,i}_k[h_i](x)\|_{L^2_x(V^r_{k \in [k_{j-1,k_j})})} &\le&  C_{r,\epsilon} |L^{\flat}|^{\epsilon} \|h_i\|_{L^2}
\end{eqnarray*} 
for $i=1,2$, finishing the proof.
\end{proof}

\section{Size and  a variation-norm size bound} \label{s.size}

We will use the following standard notion of size:

\begin{definition}
Let $j \in \{1,2,3\}$, $\P \subset \P_{\nu}$, and $f$ be a function on $\rea.$ Then
\begin{eqnarray*}
\size_j(\P,f) = \sup_{T \subset \P}\left(\frac{1}{|I_{T}|} \sum_{P \in T} |\<f,\phi_{P,j}\>|^2\right)^{1/2}
\end{eqnarray*} 
where the supremum is over all $j$-lacunary trees contained in $\P$ and where the functions $\phi_{P,j}$ are defined in Section \ref{s.discretization}.
\end{definition}

The aim of this section is to prove:

\begin{proposition} \label{p.sizebound}
Let $s > 1,$ $r > 2,$ and $\P \subset \P_{\nu}.$ Then for $j \in \{1,2,3\}$
\begin{eqnarray*}
\size_j(\P,f) &\lesssim_{r,s}& \sup_{P,P' \in \P} \ \  \sup_{I_P \subset I \subset I_{P'}} \left( \frac{1}{|I|} \int |f(x)|^s \widetilde \chi_{I}(x)^2 \ dx\right)^{1/s} 
\end{eqnarray*} 
where the inside supremum is over dyadic intervals.
\end{proposition}

We will make use of a John-Nirenberg type lemma, proven in \cite{mtt2004}

\begin{lemma} \label{john-nirenberg-lemma}

Let $\{c_P\}_{P \in \P}$ be a collection of coefficients. Let $j\in \{1,2,3\}$. For $1\le p <\infty$ let
\begin{eqnarray*}
B_{p} &=& \sup_{T \subset \P}\frac{1}{|I_T|^{1/p}} \Big\|(\sum_{P \in \T}|c_P|^2\frac{1_{I_P}}{|I_P|})^{1/2} \Big\|_{L^{p}} 
\end{eqnarray*}
where the sup  is over all $j$-lacunary trees, and define $B_{1,\infty}$ analogously. Then
\begin{eqnarray*}
B_2 &\le& C B_{1,\infty}
\end{eqnarray*}
\end{lemma}

Recall that $\nu=(j_1,j_2,e,i)$ with $0\le j_1,j_2\le 4$, $498 \le |e| \le 4002$, and $0\le i < 4000$. We will also need the following lemma:

\begin{lemma} \label{truncatebyconvlemma}
There is a Schwartz function $\zeta$ such that for each $j$-lacunary tree $T \subset \P_{\nu}$, $j\in \{1,2,3\}$, each sequence of coefficients $\{c_P\}_{P \in T}$ and each integer $k$ with $2^k \leq |I_T|$ and $-k=i \mod 4000$  we have
\begin{eqnarray*}
\sum_{P \in T: \ |I_P| \geq 2^k}c_P \phi_{P,j}(x) =\e^{2 \pi \i c(\omega_{(P_T)_j}) x}\left(2^{-k}\zeta(2^{-k} \cdot)*\Big[\e^{-2 \pi \i c(\omega_{(P_T)_j}) \cdot}\sum_{P \in T}c_P \phi_{P,j}\Big]\right)(x)
\end{eqnarray*} 
\end{lemma}

\begin{proof}
One can check that for each $P$ with $|I_P| < |I_T|$, we have 
\begin{eqnarray*}
\omega_{P,j} \subset  c(\omega_{(P_T)_j}) \pm (10 |I_P|^{-1},10000 |I_P|^{-1}) 
\end{eqnarray*}
(the sign depends on $e$ and $j$.) Choosing $\zeta$ with $\widehat{\zeta}=1$ on $(-10000,10000)$ and $\widehat{\zeta}$ supported on $(-10001,10001),$
the fact that for $P \in T$ we have $-\log_2(|I_P|)=i \mod 4000$  then gives the lemma since $2^{4000}\cdot10> 10001.$
\end{proof}

\begin{proof}[Proof of Proposition \ref{p.sizebound}]
The case $j=1,2$ of the desired conclusion is standard and we could actually get $s=1$, so the argument below (while applicable for all $j$) is only needed for $j=3$.  By Lemma \ref{john-nirenberg-lemma} it suffices to fix a $j$-lacunary tree $T \subset \P$ and show that
\begin{eqnarray} \label{sbdisplay1}
&&\frac{1}{|I_T|^{1/s}} \Big\| \big(\sum_{P \in T}|\<f,\phi_{P,j}\>|^2\frac{1_{I_P}}{|I_P|}\big)^{1/2}\Big\|_{L^s} \\
\nonumber &\lesssim_{r,s}& \quad  \sup_{P,P' \in T}\ \ \sup_{I_P \subset I \subset I_{P'}}\left( \frac{1}{|I|} \int |f(x)|^s \widetilde \chi_I(x)^2 \ dx\right)^{1/s} .
\end{eqnarray}
By dividing $T$ into maximal subtrees with top, we may assume that $P_T\in T$. Let $R$ denote the right side of \eqref{sbdisplay1}.
If $supp(f) \subset \rea \setminus 2I_T$ then for each $P \in T$
\begin{eqnarray} \label{sbdisplay-1}
|\<f,\phi_{P,j}\>|\frac{1}{|I_P|^{1/2}} 
&\le& C \left(\frac{|I_P|}{|I_T|}\right)^{M-2} \frac{1}{|I_P|} \int |f(x)| \widetilde \chi_{I_P}(x)^2 \ dx \\
\nonumber &\le& C  \left(\frac{|I_P|}{|I_T|}\right)^{M-2} R  \ \ .
\end{eqnarray}
With $M$ sufficiently large (say $M>3$), it follows that the left side of \eqref{sbdisplay1} is bounded above by
\begin{eqnarray} 
\nonumber &\le& C \frac{1}{|I_T|^{1/s}} \sum_{P \in T} \left(\frac{|I_P|}{|I_T|}\right)^{M-2} |I_P|^{1/s} R \\
\nonumber &\le& CR 
\end{eqnarray}

Thus, it remains to prove \eqref{sbdisplay1} for functions supported on $2I_T$. From this support assumption, we see that it suffices (by choosing $I=I_T$) to show
\begin{eqnarray} \label{sbdisplay2}
\Big\|\big(\sum_{P \in T}|\<f,\phi_{P,j}\>|^2\frac{1_{I_P}}{|I_P|}\big)^{1/2}\Big\|_{L^s} &\le& C_{r,s} \quad \|f\|_{L^s}.
\end{eqnarray}
By the usual Rademacher function argument, the left side of \eqref{sbdisplay2} is
\begin{eqnarray*}
&\le& \sup_{\{b_P\}_{P \in T}}  \|\sum_{P \in T}b_P \<f,\phi_{P,j}\>h_{I_P}\|_{L^s}
\end{eqnarray*}
where the supremum is over all sequences $\{b_P\}$ of $\pm1$'s on $T$ and $h_{I_P}$ is the $L^2$ normalized Haar function adapted to $I_P.$ After fixing such a sequence and using duality, we are then reduced to showing the bound
\begin{eqnarray} \label{sbdisplay3}
\|\sum_{P \in T}b_P \<g,h_{I_P}\>\phi_{P,j}\|_{L^{s'}}
 &\le& C_{r,s} \|g\|_{L^{s'}}
\end{eqnarray}
where $s'=s/(s-1).$
Recalling the definition of $\phi_{P,j}$ the left side of \eqref{sbdisplay3} is 
\begin{eqnarray} \label{sbdisplay4}
&\le& \|\sum_{\substack{P \in T \\ |I_P| \geq 2^k}}b_P \<g,h_{I_P}\>\psi_{P,j}(x)\|_{L^{s'}_x(V^r_k)}.
\end{eqnarray}
By Lemma \ref{truncatebyconvlemma}\footnote{Here we use the fact that the variation over all $k$ in \eqref{sbdisplay4} is the same as the variation restricted to $2^k \leq |I_T|$ and $-k = i \mod 4000$ which is the same as the restricted variation for convolution which is bounded by the variation over all $k$ for the convolution.} the display above is
\begin{eqnarray*}
&\le& \|2^{-k}\zeta(2^{-k} \cdot)*\big[\e^{-2 \pi \i c(\omega_{(P_T)_j}) \cdot}\sum_{P \in T}b_P \<g,h_{I_P}\>\psi_{P,j}\big](x)\|_{L^{s'}_x(V^r_k)}.
\end{eqnarray*}
By Lemma \ref{l.lepingle}, the display above is
\begin{eqnarray*} 
&\le& C_{r,s} \quad \| \sum_{P \in T}b_P \<g,h_{I_P}\>\psi_{P,j}\|_{L^{s'}} \\
&\lesssim& \|g\|_{L^{s'}(\R)}
\end{eqnarray*}
the second estimate follows from  standard Calder\'{o}n-Zygmund theory.
\end{proof}

\section{A variation-norm size increment lemma} \label{s.sizelemma}

\begin{proposition} \label{l.size-increment}
Let $\P \subset \P_\nu$ be a finite collection of tri-tiles, $\delta > 0,$ $r > 2$, and $j \in \{1,2,3\}.$ Suppose that $M$ (from the hypotheses of Theorem \ref{t.model-varnorm}) is sufficiently large depending on $\delta.$   Then for each $\alpha$ satisfying 
\begin{eqnarray*}
\size_j(\P,f) &\le& \alpha
\end{eqnarray*} 
we can find a collection of trees $\T,$ each contained in $\P$, satisfying 
\begin{eqnarray} \label{l.size-incrementconc1}
\size_j(\P \setminus \bigcup_{T \in \T}T,f)  &\le& \frac{1}{2} \alpha,
\end{eqnarray}
\begin{eqnarray} \label{l.size-incrementconc2}
\sum_{T \in \T}|I_T|  &\lesssim& \left(\frac{\|f\|_{L^{\infty}}}{\alpha}\right)^{\delta}\alpha^{-2} \|f\|_{L^2}^2 \ \ .
\end{eqnarray}
\end{proposition}

Below, we show how Proposition \ref{l.size-increment} follows from the variation-norm Bessel inequality, Theorem~\ref{t.vnbi}. The proof uses a standard stopping time argument, which we recall in order to note that our condition \eqref{thirdstrongdisjoint} in the definition of strong j-disjointness is satisfied.

Below recall that if $T$ is $i$-overlapping then for each $ j \in \{1,2,3\}\setminus\{i\}$ the sign $\epsilon_{i,j} := \sgn(c(\omega_{P_j}) - c(\omega_{(P_T)_j}))$ depends only on $i$, $j$, $e$, and not on $T$ (for details see the discussion after Definition~\ref{d.tree}).

\begin{proof}[Proof. (reduction to Bessel inequality)]
By scaling $f$ we may asssume that $\alpha=1$. 

It suffices to show that for each $i \in \{1,2,3\} \setminus \{j\}$ we could find $\T$ satisfying \eqref{l.size-incrementconc2} such that for each $i$-overlapping tree $ T \subset \P \setminus \bigcup_{T' \in \T}T'$ we have 
\begin{eqnarray} \label{l.size-incrementconc1rs}
\frac{1}{|I_{T}|} \sum_{P \in T} |\<f,\phi_{P,j}\>|^2 &\le& \frac{1}{4}.
\end{eqnarray}

Let $T_0=S_0=\emptyset.$
Suppose $T_0, \ldots, T_n$ and $S_0, \ldots, S_n$ have been chosen and set 
\begin{eqnarray*}
\P_n = \P \setminus \big(\bigcup_{k = 0}^n T_k \cup S_k\big)
\end{eqnarray*}
If there are no $i$-overlapping trees $T \subset \P_n $ violating \eqref{l.size-incrementconc1rs} then we finish by setting 
\begin{eqnarray*}
\T = \{T_k\}_{k=1}^n \cup \{S_k\}_{k=1}^n.
\end{eqnarray*}
Otherwise, if $\P_n$ contains an $i$-overlapping tree violating  $\eqref{l.size-incrementconc1rs}$ then we may choose such a tree $\widetilde T_{n+1}$ so that  $\epsilon_{i,j}c(\omega_{(P_{\widetilde{T}_{n+1}})_j})$ is maximal. We then let $T_{n+1}$ be the maximal (with respect to inclusion) $i$-overlapping tree contained in $\P_n$ which satisfies $P_{T_{n+1}} = P_{\widetilde{T}_{n+1}}$. Let $S_{n+1}$ be the maximal (with respect to inclusion) $j$-overlapping tree contained in $\P_n \setminus T_{n+1}$ which satisfies $P_{S_{n+1}}=P_{\widetilde{T}_{n+1}}$ 

Since $\P$ is finite and $T_n \neq \emptyset$, this process will eventually terminate, yielding some
\begin{eqnarray*}
\T := \{T_k\}_{k=1}^N \cup \{S_k\}_{k=1}^N.
\end{eqnarray*}
We claim that the collection $\{T_k\}_{k=1}^N$ is strongly $j$-disjoint (recall that this is defined in Definition~\ref{d.strong-j-disjoint}), and so Proposition \ref{l.size-increment} follows from Theorem \ref{t.vnbi}. It suffices to verify condition~\eqref{secondstrongdisjoint} and condition~\eqref{thirdstrongdisjoint} of Definition~\ref{d.strong-j-disjoint}. 

In the following, let $k \neq k'$, $P \in T_k$, and $P' \in T_{k'}$.

For \eqref{secondstrongdisjoint},   assume that $\omega_{P_j} \subsetneq \omega_{P'_j}$. Then $|\omega_{P'_j}| \geq 2^{4000}|\omega_{P_j}|$ which implies that $\epsilon_{i,j}c(\omega_{(P_{T_{k}})_j}) >  \epsilon_{i,j}c(\omega_{(P_{T_{k'}})_j})$ and so $k < k.'$ But, since  $3\omega_{(P_{T_k})_j} \subset 30|e|\omega_{P_j} \subset 3\omega_{P'_j}$ and $P' \notin S_k$
we must have $I_{P'} \cap I_{T_{k}} = \emptyset$. 

Now,  to see condition \eqref{thirdstrongdisjoint}, by symmetry it suffices to show that $P'_j \not \le (P_{T_k})_j$. First suppose that $P'_j = (P_{T_{k}})_j$, or equivalently $P'=P_{T_k}$. Then for each $P \in T_k$ we have $P_i \leq P'_i \leq (P_{T_{k'}})_i$ and so we must have $k < k'$ or else every element of $T_k$ would have already been chosen in $T_{k'}$. But, if $k < k'$ then we would have $P' \in T_k$, contradicting $P' \in T_{k'}$. Now, suppose that $P'_j < (P_{T_{k}})_j$. Then, as in the verification of \eqref{secondstrongdisjoint}, $\epsilon_{i,j}c(\omega_{(P_{T_{k}})_j}) >  \epsilon_{i,j}c(\omega_{(P_{T_{k'}})_j})$ and so $k < k.'$ But, the fact that $P' \notin S_k$ contradicts $P'_j \leq (P_{T_{k}})_j$.
\end{proof}

\section{A variation-norm Bessel inequality} \label{s.varnorm-bessel}

In this section, we fix $\sigma>0$ and assume that the order $M$ of the wave packets (from the hypotheses of Theorem \ref{t.model-varnorm}) is sufficiently large depending on $\sigma$. Our goal here is to prove the following variation-norm Bessel inequality:

\begin{theorem} \label{t.vnbi}
Let $\T$ be a collection of strongly $j$-disjoint trees, such that 
\begin{eqnarray}\label{e.size-hypothesis}
\sup_{I \ dyadic} (\frac{1}{|I|} \sum_{\substack{P \in T \\ I_P \subset I}}  |\<f,\phi_{P,j}\>|^2)^{1/2} &\le& 1 \quad \le \quad 2 (\frac{1}{|I_T|} \sum_{P \in T}  |\<f,\phi_{P,j}\>|^2)^{1/2}    \ \ ,
\end{eqnarray}
for each $T\in \T$.   Let $\P=\bigcup_{T\in \T} T$. Then
\begin{eqnarray*}
\sum_{P \in \P}   |\<f,\phi_{P,j}\>|^2 &\lesssim_{\sigma}  \|f\|_{L^{\infty}}^{\sigma} \| f\|_{L^2}^2. 
\end{eqnarray*}
\end{theorem}

 As in \cite{dtt2008}, we prove Theorem~\ref{t.vnbi} via a sequence of  reductions.

\subsection{Proof of Theorem~\ref{t.vnbi}, reduction 1}
Thanks to Lemma~\ref{l.self-improving} below, Theorem~\ref{t.vnbi} follows immediately from the following proposition:
\begin{proposition}\label{p.delta-loss}Assume $\P$ and $\T$ as in Theorem~\ref{t.vnbi}. Then for all $\delta>0$
\begin{eqnarray}\label{e.delta-loss}
\sum_{P \in \P} |\<f,\phi_{P,j}\>|^2 &\lesssim_\delta&  \|N_{\T}\|_{L^{\infty}}^\delta 
\int |f(x)|^2 \widetilde \chi_I(x)^{10} \ dx \ \ ,
\end{eqnarray}
if $I_T \subset I$ dyadic for all $T \in \T$.
\end{proposition} 
Lemma~\ref{l.self-improving} below in turn is a result from \cite{dtt2008} where it was proved using a series of interesting Lemmas. To keep the current paper self-contained, we'll sketch  a direct proof, which simplifies some arguments in \cite{dtt2008}. To formulate the lemma, we first fix some notations. For $\S\subset \T$ let $N_{\S}$ denote $\sum_{T\in \S} 1_{I_T}$, and define
\begin{eqnarray*}
\|\S\|_{BMO}:=\sup_{I \ dyadic} \frac 1 {|I|} \sum_{T\in \S: I_T\subset I} |I_T| \ \ .
\end{eqnarray*}

\begin{lemma}\label{l.self-improving} Let $A,B>0$ and $0<\delta<1$. Let $\T$ be a collection of trees. If for every subset $\S$ of $\T$ it holds that $\|N_{\S}\|_1 \le A \|N_{\S}\|_\infty^\delta$ and $\|\S\|_{BMO}   \le B  \|N_{\S}\|_\infty^\delta$
then
\begin{eqnarray*}
\|\T\|_{BMO} \lesssim_\delta B^{1/(1-\delta)} \quad , \quad \|N_{\T}\|_1 \lesssim_\delta A B^{\delta/(1-\delta)}
\end{eqnarray*}
\end{lemma}

\proof[Proof of Lemma~\ref{l.self-improving}] We first show that $\|\T\|_{BMO} \le (3B)^{1/(1-\delta)}$.   It suffices to show that for every dyadic interval $I_0$ it holds that
\begin{eqnarray}\label{e.forestBMO}
\frac 1 {|I_0|} \sum_{T\in \T: I_T\subset I_0} |I_T| \le \frac 1 2\|\T\|_{BMO} +  3^{\delta/(1-\delta)} B^{1/(1-\delta)} \ \ 
\end{eqnarray}
Fix $I_0$. Let
$\S$ contains all elements $T\in \T$ such that $I_T\subset I_0$ and the set $\{S\in \T: I_T\subset I_S \subset I_0\}$ contains  at most $ (3B)^{1/(1-\delta)}$ elements.  Clearly $\|N_{\S}\|_\infty \le (3B)^{1/(1-\delta)}$, therefore by the given assumption we have
\begin{eqnarray}\label{e.SBMO}
\|\S\|_{BMO} \le  B(3B)^{\delta/(1-\delta)}  = 3^{\delta/(1-\delta)} B^{1/(1-\delta)}\ \ . 
\end{eqnarray}
Let $\J$ be the set of maximal dyadic intervals $J \subset I_0$ such that the set $\{S\in \T: J\subset I_S \subset I_0\}$ contains more than $(3B)^{1/(1-\delta)}$ elements. Clearly, for every $T\in \T\setminus \S$ such that $I_T\subset I_0$, $I_T$ is contained in one of these $J$'s. It follows that
\begin{eqnarray}
\label{e.forestBMO1}  \frac 1 {|I_0|} \sum_{\substack{T\in \T \setminus \S\\ I_T\subset I_0}} |I_T| \quad \le \quad \frac 1 {|I_0|} \sum_{J \in \J} \sum_{\substack{T\in \T\\ I_T\subset J}} |I_T| &\le& \frac{\|\T\|_{BMO} }{|I_0|} \sum_{J\in \J} |J| \ \ .
\end{eqnarray}
By maximality of $J$, there exists $T\in \T$ such that $I_T=J$. Let $\S_J$ denote the collection of such $T$, then $\|\S_J\|_{BMO}=\|N_{\S_J}\|_{\infty} = |\S_{J}|$, therefore using the given assumption we obtain $\|N_{\S_J}\|_{\infty} \le B^{1/(1-\delta)}$. For every $x\in J$ it follows that 
\begin{eqnarray*}N_{\S}(x) &\ge& \sum_{T\in \T: J\subsetneq I_T \subset I_0} 1_{I_T}(x)  \quad \ge \quad \Big(\sum_{T\in \T: J\subset I_T \subset I_0} 1_{I_T}(x)\Big) - B^{1/(1-\delta)} \\
&\ge& (3^{1/(1-\delta)}-1) B^{1/(1-\delta)} \quad > \quad 2 \cdot 3^{\delta/(1-\delta)} B^{1/(1-\delta)} \ \ .
\end{eqnarray*}
Together with \eqref{e.forestBMO1}, we obtain
\begin{eqnarray*}
\frac 1 {|I_0|} \sum_{T\in \T \setminus \S: I_T\subset I_0} |I_T|   
&<& \frac{\|\T\|_{BMO}}{|I_0|} \frac{\|N_{\S}\|_{L^1(I_0)}}{2 \cdot 3^{\delta/(1-\delta)}} \quad \le \quad \|\T\|_{BMO}  \frac{\|\S\|_{BMO}}{2 \cdot 3^{\delta/(1-\delta)}} \ \ .
\end{eqnarray*}
Using \eqref{e.SBMO},  \eqref{e.forestBMO} immediately follows, completing  the proof of $\|\T\|_{BMO} \lesssim_\delta B^{1/(1-\delta)}$. 

We now free the variables $I_0$, $\S$, $\J$ to be used for other purposes below. 

Fix a large   constant $C>0$ to be chosen later. Let $\S$ contain all  $T\in \T$ such that $I_T$ is not a subset of $\{x: N_{\T}(x) > C B^{1/(1-\delta)}\}$. It is clear that $\|N_{\S}\|_\infty \le C B^{1/(1-\delta)}$, so by the given hypothesis $\|N_{\S}\|_1 \le C^\delta AB^{\delta/(1-\delta)}$. It suffices to show that 
\begin{eqnarray}\label{e.NTabsorb}
\|N_{\T \setminus \S}\|_1 = O_\delta(C^{-1}\|N_{\T}\|_1)  \ \ . 
\end{eqnarray}
Indeed, from \eqref{e.NTabsorb} by choosing $C$ large we obtain $\|N_{\T \setminus \S}\|_1 \le 1/2 \|N_{\T}\|_1$, thus $\|N_{\T}\|_1 \le 2 \|N_{\S}\|_1$ which implies the desired estimate.

Let $\J$ be the collection of  maximal dyadic subintervals of $\{x: N_{\T}(x) > C B^{1/(1-\delta)}\}$. It follows that if $T\in \T\setminus \S$ then $I_T$ is a subset of some element of $\J$. Therefore
\begin{eqnarray*}
\|N_{\T\setminus \S}\|_1  &\le& \sum_{J\in \J} \sum_{T\in \T: I_T\subset J} |I_T|  \quad \le \quad \sum_{J\in \J} |J| \|\T\|_{BMO}\\
&\le& \|\T\|_{BMO} |\{x: N_{\T}(x) > C B^{1/(1-\delta)}\}| \quad \le \quad \|\T\|_{BMO} \frac{\|N_{\T}\|_1}{CB^{1/(1-\delta)}} \ \ .
\end{eqnarray*}
Since  $\|\T\|_{BMO}=O_\delta(B^{1/(1-\delta)})$, we obtain $\|N_{\T\setminus \S}\|_1 = O_\delta(C^{-1} \|N_{\T}\|_1)$, as desired. \endproof


\subsection{Proof of Theorem~\ref{t.vnbi}, reduction 2}
We first note that   \eqref{e.delta-loss} follows from the unweighted version where the factor $\widetilde \chi_I^{10}$ is not on the right hand side. Indeed, writing $\<f,\phi_{P,j}\> = \<f\widetilde \chi_I^{10}, \widetilde \chi_I^{-10}\phi_{P,j}\>$ and using the fact that $\widetilde \chi_I^{-10}(x)$ is also a polynomial in $x$ (which implies that $\widetilde \chi_I^{-10}\psi_{P,j}$ is still a wave packet adapted to $P_j$ of order sufficiently large, recall also that $\phi_{P,j}$ and $\psi_{P,j}$ are the same if $j=1,2$ and related by a variational factor if $j=3$), \eqref{e.delta-loss} follows from applying the unweighted version to $f\widetilde \chi_I^{10}$ and the rescaled wave packets.


We now show that the unweighted \eqref{e.delta-loss} follows from the following proposition.

\begin{proposition}\label{p.delta-loss-weaktype} Let $\T$ be strongly $j$-disjoint. Let $\P=\bigcup_{T \in \T} T$. Then for every $L\ge \|N_{\T}\|_\infty$ there exists $\P^\ast\subset \P$ with the following two properties:
\begin{eqnarray}
\label{e.delta-loss-weak1}
\sum_{P \in \P\setminus \P^\ast} |\<f,\phi_{P,j}\>|^2 &\lesssim_{\delta}& L^\delta 
\|f\|_2^2 \ \ , \\  
\label{e.delta-loss-weak2}
|\bigcup_{P\in \P^\ast} I_P| &\lesssim_\delta&  L^{-1} \sum_{T\in \T} |I_T| \ \ .
\end{eqnarray}
\end{proposition}
Indeed, apply Proposition~\ref{p.delta-loss-weaktype} with $L=C\|N_T\|_\infty$ for a sufficiently large $\delta$-dependent constant $C$. Now, to get (the unweighted) \eqref{e.delta-loss} it suffices to show 
\begin{eqnarray}\label{e.delta-loss-weak3}
\sum_{P\in \P^\ast} |\<f,\phi_{P,j}\>|^2 \le \frac 1 2 \sum_{P\in \P} |\<f,\phi_{P,j}\>|^2  \ \ .
\end{eqnarray}
Let $I_{P_0}$ be a maximal interval in $\{I_P, P\in \P^\ast\}$, and remove from $\P^\ast$ all tri-tiles $P$ such that $P$ is in the same tree as $P_0$ and $I_P\subset I_{P_0}$. We repeat this process with what is left of $\P^\ast$. This algorithm gives a collection of tree $\T^\ast$ such that $\{I_T, T\in \T^\ast\}$ cover $\P^\ast$ while $\sum_{T\in \T^\ast} 1_{I_T} \le \sum_{T\in \T} 1_{I_T}$. Now, using \eqref{e.size-hypothesis} and \eqref{e.delta-loss-weak2},  \eqref{e.delta-loss-weak3} follows from the following sequence of estimates and choosing $C$ large:
\begin{eqnarray*}
\sum_{P\in \P^\ast} |\<f,\phi_{P,j}\>|^2 &\le& \sum_{T\in \T^\ast} |I_T| \quad \le  \quad \|N_{\T^\ast}\|_\infty |\bigcup_{P\in \P^\ast} I_P|  \quad \lesssim \\
&\lesssim_\delta& (L/C)L^{-1} \sum_{T\in \T} |I_T|  \quad \le \quad \frac 2 C \sum_{P\in \P} |\<f,\phi_{P,j}\>|^2 \ \ .
\end{eqnarray*}

\subsection{Proof of Theorem~\ref{t.vnbi}, reduction 3}

In  this section we reduce Proposition~\ref{p.delta-loss-weaktype} to the following more technical result. We first fix some notations. Given $A>1$ and $d \in \{0,1,2\}$,    a collection of intervals $\mathcal{I} \subset \calG_0$ is   $(A,d)$-sparse if   
\begin{itemize}
\item for each $I \in \mathcal{I}$, $AI$ is $d$-regular (see Section~\ref{s.terms});
\item for each $I,I' \in \mathcal{I}$ with $|I| > |I'|$, we have $|I| \geq 2^{100A}|I'|$;
\item for each $I,I' \in \mathcal{I}$ with $|I| = |I'|$, we have $\dis(I,I') \geq 100A|I'|$.
\end{itemize} 

\begin{proposition} \label{p.vnbi-sparse}
Let $A, L, \eta \in (1,\infty)$ and $\epsilon>0$. 
Let $\T$ be a collection of strongly $j$-disjoint trees with
$\|N_{\T}\|_{L^{\infty}} \leq L$. 
Let $\P = \bigcup_{T \in \T}T,$ and suppose that $\{I_P : P \in \P\} \cup \{I_T : T \in \T\}$
is $(A,d)$ sparse. Then, there exists $\P^* \subset \P$  such that
\begin{eqnarray*}
|\bigcup_{P \in \P^*} I_P| &\lesssim_{\eta}& (A^{-\eta} + L^{-\eta}) \sum_{T \in \T}|I_T| \ \ ,  \ \ and \\
\sum_{P \in \P \setminus \P^*} |\<f,\phi_{P,j}\>|^2 &\lesssim_{\eta,\epsilon}& ((AL)^{\epsilon} + L^8A^{3-\eta})^2 \|f\|_{L^2}^2.
\end{eqnarray*} 
\end{proposition}
Assuming Proposition~\ref{p.vnbi-sparse} we will prove Proposition~\ref{p.delta-loss-weaktype}. 

Let $\T$ be $j$-strongly disjoint and $\P=\bigcup_{T\in \T} T$.  

By a simple pigeonholing argument, given any $A > 1$ we may partition $\P$ into subsets $\P_1,\dots, \P_L$ where $L=O(A^2)$, with the following property: for each $1\le k \le L$  there exists $d\in \{1,2,3\}$ such that $\{ I_P, P\in \P_k\}$ is $(A,d)$-sparse. 

This partition also lead to a partition of each $T \in \T$, therefore $\P_k$ is also the union of a collection $\T_k$ of $j$-strongly disjoint trees with $\|N_{\T_k}\|_\infty \le \|N_{\T}\|_\infty \le L$. Each tree in $\T_k$ could be further decomposed into subtrees such that: each of the new subtrees contains its own top, and the top intervals of the subtrees are disjoint. We obtain $\T'_k$ a collection of trees with top, which is still $j$-strongly disjoint, furthermore $\|N_{\T'_k}\|_\infty \le \|N_{\T_k}\|_\infty \le L$. 

We are now in a position to apply Proposition~\ref{p.vnbi-sparse} for $\T'_k$, producing  $\P_k^\ast \subset \P_k$.
Letting $\P^\ast = \bigcup_{1\le k \le L} \P_k^\ast$ and using $L =O(A^2)$ it follows that
\begin{eqnarray*}
|\bigcup_{P\in \P^\ast} I_P| &\lesssim_\eta& (A^{2-\eta} + A^2 L^{-\eta})\sum_{T\in \T} |I_T| \\
\sum_{P\in \P \setminus \P^\ast} |\<f,\phi_{P,j}\>|^2 
&\lesssim_{\eta,\epsilon}& (A^{1+\epsilon} L^{\epsilon} + L^8 A^{4-\eta})^2 \|f\|_2^2
\end{eqnarray*}
The desired estimates for $\P^\ast$ follows by letting $\epsilon = \delta/2$, $A=L^{\epsilon/(1+\epsilon)}$, and $\eta$ large.
 
\subsection{Proof of Theorem~\ref{t.vnbi}, reduction 4}
Let $\I=\{I_T: T\in \T\}$. Let $D_\eta>0$ to be chosen later (depending only on $\eta$). For each $I\in \I$ consider the $D_\eta (A^{-\eta}+L^{-\eta})|I|$ neighborhood of its endpoints, i.e. the set of $x$ such that $\dist(x,\partial I) \le D_\eta (A^{-\eta}+L^{-\eta})|I|$. Let $E_1$ be the union of these neighborhoods over $I\in \I$, and let $E_2=\{\sum_{I\in \I} (\mathcal M1_I)^2 > L^2\}$.
We then let $\P^\ast$ be the set of all $P\in \P$ such that $I_P\subset E_1\cup E_2$. Using the Fefferman--Stein maximal inequality, it follows that
\begin{eqnarray*}
|\bigcup_{P\in \P^\ast} I_P|   &\lesssim_\eta& (A^{-\eta} + L^{-\eta}) \sum_{I\in \I} |I| + L^{-(4+2\eta)}\|\sum_{I\in \I} (\mathcal M1_I)^2\|_{2+\eta}^{2+\eta} \\
&\lesssim_\eta& (A^{-\eta} + L^{-\eta})\|N_{\T}\|_1 \ \ .
\end{eqnarray*}
Let $\T_2=\{T\in \T: I_T \not\subset E_1\cup E_2\}$ and let $\I_2$ denote the set of top intervals of $\T_2$.  We now show that  $\|\sum_{T\in \T_2} (\mathcal M[1_{I_T}])^2\|_\infty \lesssim_\eta L^{4}$, and this will allow us to reduce Proposition~\ref{p.vnbi-sparse} to Proposition~\ref{p.lastred} below.
Since for each $I \in \I_2$ there are at most $L$ elements of $\T$ with $I_T=I$, it suffices to show that
\begin{eqnarray*}
\sum_{I\in \I_2} (\mathcal M[1_{I}])(x)^2 &\lesssim_\eta& L^{3}
\end{eqnarray*}
uniform over $x\in \R$, which we fix below.
By further dividing $\I_2$ it suffices to prove that $\sum_{I\in \I_3} (\mathcal M[1_{I}])(x)^2 \lesssim_\eta L^2$
for every $\I_3\subset \I_2$ with the following property: if $I, I'\in \I_3$ and $|I'|<|I|$ then $|I'|\le  2^{-L}|I|$. By further dividing $\I_3$ we may assume that one of the following situations occur: (i) $x\in I$ for all $I\in \I_3$; (ii) $x$ is on the left of $I$ for all $I\in \I_3$;
(iii) $x$ is on the right of $I$ for all $I\in \I_3$.

Now,  the desired estimate is clear for (i), so by symmetry we only consider situation (ii). By monotonicity we may assume further that $x$ is the left endpoint of some $J\in \I_3$. By definition of $E_1$ it follows that for every $I\in \I_3-\{J\}$ we have $\dist(x, I) \gtrsim L^{-\eta}\max(|J|,|I|)$. Using the $(A,d)$ sparseness of $\I_3$, it follows that
\begin{eqnarray*}
\sum_{I\in \I_3}  (\mathcal M[1_{I}])(x)^2 &\lesssim& \sum_{|I|=|J|} (\mathcal M 1_{I})(x)^2 \quad + \sum_{|I| \le 2^{-L} |J|} (\mathcal M 1_{I})(x)^2 + \sum_{|I| \ge 2^L |J|} (\mathcal M 1_{I})(x)^2\\
&\lesssim_\eta& 1 \quad +\quad L^\eta 2^{-L} \quad + \quad \inf_{x\in J} \sum_{I\in \I_3: |I| \ge 2^L |J|} (\mathcal M 1_{I})(x)^2\\
&\lesssim&   L^2 \qquad \text{(using the definition of $E_2$).}
\end{eqnarray*}

\begin{proposition}  \label{p.lastred}
Let $A, L, \eta > 1$ and $\epsilon>0$. Let $\T$ be strongly $j$-disjoint with
\begin{eqnarray}  \label{e.masq}
\|\sum_{T \in \T}(\ma[1_{I_T}])^2\|_{L^{\infty}} &\le& L\ \ .
\end{eqnarray}
Let $\P = \bigcup_{T \in \T}T$. Assume that  $\{I_P : P \in \P\} \cup \{I_T : T \in \T\}$
is $(A,d)$ sparse, and
\begin{eqnarray} \label{e.intcond}
\sup_{x \in I_P} \dis(x,\partial I_T) \geq D_{\eta} A^{-\eta}|I_T|
\end{eqnarray}
for all $P \in \P$, $T \in \T$.   Then for $D_{\eta}$ sufficiently large depending on $\eta$ it holds that
\begin{eqnarray*} 
\sum_{P \in \P}|\<f,\phi_{P,j}\>|^2 &\lesssim_{\eta,\epsilon}& ((AL)^{\epsilon} + L^2A^{3-\eta})^2 \|f\|_{L^2}^2.
\end{eqnarray*}
\end{proposition}

\subsection{Proof of Proposition~\ref{p.lastred}}

For convenience of notation, assume without loss of generality that $A^2 L$ is an integer. 
By duality, it suffices to show 
\begin{eqnarray}\label{e.dual-red4}
\|\sum_{P \in \P} b_P \phi_{P,j}\|_{L^2} &\lesssim_{\eta,\epsilon}& (AL)^{\epsilon} + L^2A^{3-\eta} 
\end{eqnarray}
for every  sequence $\{b_P\}_{P \in \P}$ such that $\|b\|_{\ell^2(\P)}=1$, which we will fix below.

Let $\calJ = \{I_T : T \in \T\}$ and let $\calJ_A = \{(I_{T})_A : T \in \T\}$ where $(I_T)_A$ is an interval in $\calG_d$ (guaranteed by $(A,d)$ sparsity) such that $AI_T \subset (I_T)_A \subset 3AI_T$.

From the $(A,d)$ sparsity, it is clear that the map from $I_T \rightarrow (I_T)_A$ is bijective from $\calJ$ to $\calJ_A$ and that if $I_T \subsetneq I_{T'}$ then $(I_T)_A \subsetneq (I_{T'})_A$. 

We now decompose $\calJ_A$ into ``layers''. Let $\calJ_{A,1}$ be the set of maximal intervals in $\calJ_A$ and for $m \geq 1$ let $\calJ_{A,m+1}$ be the set of maximal intervals in $\calJ_{A} \setminus \bigcup_{n=1}^m \calJ_{A,n}$. 
 
Now, since  $(I_T)_A \subset 3AI_T$ for each $T$, using \eqref{e.masq} we have
\begin{eqnarray*}
\|\sum_{J \in \calJ_A} 1_J\|_{L^{\infty}} &\le&  16A^2 \|\sum_{T} (\mathcal M[1_{I_T}])^2\|_\infty  \quad \le \quad 16A^2 L \ \ .
\end{eqnarray*}
Thus, $\calJ_{A,1}, \ldots, \calJ_{A,16A^2L}$ partition $\calJ_A$. Letting $\calJ_{m} = \{J \in \calJ : (J)_A \in \calJ_{A,m}\}$ it follows that $\calJ_1, \ldots \calJ_{16A^2L}$ partition $\calJ$.  Thanks to $(A,d)$ sparsity of $\calJ$ again, this partition is consistent with the usual set inclusion ordering in $\calJ$, in the sense that if $J\in \calJ_m$, $J'\in \calJ_n$, and $J\subsetneq J'$ then $m > n$.

For $J \in \calJ$ let $m$ be such that $J\in \calJ_m$, and define
\begin{eqnarray*}
\P_J &:=& \{P \in \P : I_P = J\}
\end{eqnarray*}
\begin{eqnarray*}
\P_{<J} &:=& \{P \in \P : I_P \subsetneq J \text{\ but \ } I_P \not\subset J' \text{ \ for all \ } J' \in \bigcup_{m'>m}\calJ_{m'}\}.
\end{eqnarray*}
We obtain the following partition of $\P$:
\begin{eqnarray}\label{e.Ppartition}
\P &=& \bigcup_{1 \leq m \leq 16A^2L} \bigcup_{J \in \calJ_m} \P_J \cup \P_{<J} \ \ .
\end{eqnarray}
By definition of $\phi_{P,j}$, it is clear that \eqref{e.dual-red4} will follow from the following   estimates
\begin{eqnarray} 
\label{e.red41}
\|\sum_{J \in \calJ} \sum_{P \in \P_J: \  |I_P| \geq 2^k} b_P \psi_{P,j}(x)\|_{L^2_x(V^r_k)}  &\lesssim_\eta&  1 + \log^2(AL) + A^{-\eta}L   \\
\label{e.red42}
\|\sum_{J \in \calJ} \sum_{P \in \P_{<J}: \  |I_P| \geq 2^k}  b_P \psi_{P,j}(x)\|_{L^2_x(V^r_k)}  &\lesssim_{\eta,\epsilon}&  (AL)^{\epsilon} + L^2A^{3 - \eta}   \ \ .
\end{eqnarray}

\subsection{Proof of \eqref{e.red41}} \label{polrfsection}
Recall that $\|b\|_{\ell^2(\P)}=1$.  Recall that $\psi_{P,j}$ is a wave function   of order $M$, which is assumed sufficiently large compared to $\eta$.
We first estimate the error term 
\begin{eqnarray*}
E(x) &:=& \sum_{J\in \calJ: \   x \not\in (J)_A}  \sum_{P \in \P_J} |b_P \psi_{P,j}(x)| \ .
\end{eqnarray*}

\begin{lemma} \label{l.E(x)-source} It holds that
\begin{eqnarray*}
\|\sum_{J \in \calJ} \sum_{P \in \P_J: \  |I_P| \geq 2^k} b_P \psi_{P,j}(x)\|_{V^r_k} &\le& \|\sum_{m=1}^{n} \sum_{J \in \calJ_m} \sum_{P \in \P_J} b_P \psi_{P,j}(x)\|_{V^r_n} + 2 E(x) \ .
\end{eqnarray*}
\end{lemma}

\begin{proof}
Using the triangle inequality and the definition of $E(x)$, the left hand side of the desired estimate is bounded above by
\begin{eqnarray*}
&\le& \|\sum_{J \in \calJ: \ x \in (J)_A} \ \  \sum_{P \in \P_J: \  |I_P| \geq 2^k} b_P \psi_{P,j}(x)\|_{V^r_k} + E(x) \\
&=& \|\sum_{J \in \calJ: \ x \in (J)_A, \ |J| \ge 2^k} \ \  \sum_{P \in \P_J} b_P \psi_{P,j}(x)\|_{V^r_k} + E(x) \ \ .
\end{eqnarray*}
Now, the intervals  $(J)_A$ that contain $x$ are nested, with larger interval belongs to some $\calJ_{A,m}$ with smaller $m$, thus we could bound the last display by
\begin{eqnarray*}
&\le& \|\sum_{m=1}^n \sum_{J\in \calJ_m: x\in (J)_A} \sum_{P \in \P_J} b_P \psi_{P,j}(x)\|_{V^r_n} \quad + \quad E(x) \\
&\le& \|\sum_{m=1}^{n} \sum_{J \in \calJ_m} \sum_{P \in \P_J} b_P \psi_{P,j}(x)\|_{V^r_n}
\quad +\quad 2E(x)
\end{eqnarray*}
finishing the proof.
\end{proof}

\begin{lemma} \label{l.E(x)} It holds that
\begin{eqnarray*} 
\|E(x)\|_{L^2} &\lesssim_\eta& A^{-\eta} L \ .
\end{eqnarray*}
\end{lemma}

\begin{proof}

We note that any $T \in \T$  contributes at most $O(1)$ tri-tiles to each $\P_J$ and such a contribution would necessitate $J \subset I_T$. Thus, $|\P_J| \lesssim \|N_{\T}\|_{L^\infty} \le L$, so
\begin{eqnarray*}
\|\sum_{J \in \calJ} \sum_{P \in \P_J}  |I_P|^{1/2} |\psi_{P,j}|
\|_{L^{\infty}} &\lesssim& L \|\sum_{J \in \calJ}  (\ma[1_J])^2 \|_{L^{\infty}} \quad \lesssim \quad L^2 \ \ .
\end{eqnarray*}

We also have
\begin{eqnarray*}
\|1_{\rea \setminus AI_P}\psi_{P,j}\|_{L^1} &\lesssim_M& A^{1-M} |I_P|^{1/2}
\end{eqnarray*}
which gives
\begin{eqnarray*}
\|\sum_{J \in \calJ: \ x\not\in (J)_A}  \sum_{P \in \P_J} |I_P|^{-1/2} |b_P^2 \psi_{P,j}(x)|
\|_{L^{1}} &\lesssim_M&  A^{1-M } \|b\|_{\ell^{2}(\P)}^2 \quad = \quad A^{1-M} \ \ .
\end{eqnarray*}

Choosing $M$ sufficiently large, depending on $\eta$, the desired bound for $\|E\|_2$ follows by an application of Cauchy-Schwarz.
\end{proof}

Applying Lemma \ref{nonmaxbessellemma}, we see that for any sequence $\{\epsilon_m\}_{m=1}^{16A^2L} \subset \{1,-1\}$ we have
\begin{eqnarray*}
\|\sum_{m=1}^{16A^2L} \sum_{J \in \calJ_m} \sum_{P \in \P_J}  \epsilon_m b_P \psi_{P,j}\|_{L^2} &\lesssim& (1 + \log(L))  
\end{eqnarray*}
therefore, by Lemma \ref{vnrmlemma}
we have
\begin{eqnarray*}
\|\sum_{m=1}^{n} \sum_{J \in \calJ_m} \sum_{P \in \P_J} b_P \psi_{P,j}(x)\|_{L^2_x(V^r_n)} &\lesssim& [1 + \log(16A^2L)] \cdot [1 + \log(L)]
\end{eqnarray*}
which, combined with Lemma \ref{l.E(x)} and Lemma~\ref{l.E(x)-source}, gives \eqref{e.red41}.

\subsection{Proof of \eqref{e.red42}} \label{polrssection}

Here, we use two error terms 
\begin{eqnarray*}
E_1(x) &=& \sum_{m=1}^{16A^2L} \sum_{\substack{J \in \calJ_m \\ x \not\in J}} \sum_{P \in \P_{<J}} |b_P \psi_{P,j}(x)| \\
E_2(x) &=& \sum_{m = 2}^{16A^2L} \ \  \sum_{\substack{J \in \calJ_m \\ x \in J}} \ \  \sum_{m' < m} \ \  \sum_{\substack{P \in \P_{<J^{m'}} \\ |I_P| < |J|}} |b_P \psi_{P,j}(x)|
\end{eqnarray*}
where, if $J \in \calJ_m$ and $m'<m$ then we let $J^{m'}$ denote the unique element of $\calJ_{m'}$ such that $(J)_A \subset (J^{m'})_A$.

\begin{lemma} \label{l.E1(x)E2(x)} It holds that
\begin{eqnarray*}
\|\sum_{J \in \calJ} \sum_{\substack{P \in \P_{<J} \\ |I_P| \geq 2^k}} b_P \psi_{P,j}(x)\|_{V^r_k} &\lesssim&   E_1(x) + E_2(x) +  \|\sum_{m =1}^{n} \sum_{J \in \calJ_m} \sum_{P \in \P_{<J}} b_P \psi_{P,j}(x)\|_{V^r_n} \\  
&+& \big(\sum_{m, J \in \calJ_m}\| \sum_{P \in \P_{<J}: \ |I_P| \geq 2^k} b_P \psi_{P,j}(x)\|_{V^r_k}^2 \big)^{1/2} \ \ .
\end{eqnarray*}
\end{lemma}
Remark: A simpler  analogue of Lemma~\ref{l.E1(x)E2(x)} was considered in \cite[Lemma 12.2]{dtt2008}.  Our Lemma~\ref{l.E1(x)E2(x)} (and the following Lemma~\ref{l.E2(x)}) in fact fills  in a small gap in \cite[Lemma 12.2]{dtt2008}, where an error term similar to $E_2$ was not treated.

\begin{proof}
By the triangle inequality
\begin{eqnarray*}
\|\sum_{J \in \calJ} \ \  \sum_{\substack{P \in \P_{<J} \\ |I_P| \geq 2^k}} b_P \psi_{P,j}(x)\|_{V^r_k}  
&\le& \|\sum_{\substack{J \in \calJ\\ x\in J}} \ \  \sum_{\substack{P \in \P_{<J} \\ |I_P| \geq 2^k}}  b_P \psi_{P,j}(x)\|_{V^r_k} + E_1(x)
\end{eqnarray*}

Let $J_1 \subsetneq \ldots \subsetneq J_N$ be the (nested) intervals in $\calJ$ that contain $x$. Choose $k_1, \ldots, k_N$ so that $2^{k_l} = |J_l|$, which (together with $N$) are functions of $x$. Then, the first term on the right of the last display could be rewritten as
\begin{eqnarray*}
&=&  \|\sum_{1\le \ell \le N}\ \  \sum_{\substack{P \in \P_{<J_\ell} \\ |I_P| \geq 2^k}} b_P \psi_{P,j}(x)\|_{V^r_k} 
\quad \le \quad \|\sum_{\ell} \sum_{\substack{P \in \P_{<J_\ell} \\   |I_P| \ge 2^k, \ |I_P| \ge |J_{\ell+1}|}}  b_P \psi_{P,j}(x)\|_{V^r_k} + E_2(x)
\end{eqnarray*}
where, for the inequality above, we use the fact that $J_{\ell + 1} \subsetneq J_\ell$ and so $(J_{\ell+1})_A \subsetneq (J_{\ell})_A.$
Using a long jump/short jump decomposition of the variation-norm, the first term on the right side of the inequality above is $\le \mathcal A_1 + \mathcal A_2$, where
\begin{eqnarray*} 
\mathcal A_1&=& \|\sum_{\ell}\quad \sum_{\substack{P \in \P_{<J_\ell}: \  |I_P| \geq 2^{k_n}, \ |I_P| \geq 2^{k_{\ell+1}}}} b_P \psi_{P,j}(x)\|_{V^r_n} \\ 
\mathcal A_2  &=& 2\  (\sum_{n} \| \sum_{\ell} \sum_{\substack{P \in \P_{<J_\ell} \\ |I_P| \geq 2^k, \ 
|I_P| \geq 2^{k_{\ell+1}}}} b_P \psi_{P,j}(x) \|^2_{V^r_{k_{n+1} \leq k < k_{n}}})^{1/2}.
\end{eqnarray*}
It is clear that
\begin{eqnarray*}
\mathcal A_1 &=& \|\sum_{\ell < n} \sum_{\substack{P \in \P_{<J_\ell} \\ |I_P| \geq 2^{k_{\ell+1}}}} b_P \psi_{P,j}(x)\|_{V^r_n} \quad \le \quad \|\sum_{\ell < n} \sum_{\substack{P \in \P_{<J_\ell}}} b_P \psi_{P,j}(x)\|_{V^r_n} + E_2(x) \\
&=& \|\sum_{m =1}^{n} \sum_{\substack{J \in \calJ_m \\ x \in J}} \sum_{P \in \P_{<J}} b_P \psi_{P,j}(x)\|_{V^r_n} + E_2(x)\\
&\le& \|\sum_{m =1}^{n} \sum_{J \in \calJ_m} \sum_{P \in \P_{<J}} b_P \psi_{P,j}(x)\|_{V^r_n} + E_1(x) + E_2(x) \ \ , \\
\mathcal A_2 &=& 2 \Big(\sum_{n} \Big\| \sum_{\substack{P \in \P_{<J_n} \\ |I_P| \geq 2^k}} b_P \psi_{P,j}(x) + \sum_{\ell<n} \sum_{\substack{P \in \P_{<J_\ell} \\ |I_P| \geq 2^{k_{\ell+1}}}} b_P \psi_{P,j}(x) \Big\|^2_{V^r_{k_{n+1} \leq k < k_{n}}}  \Big)^{1/2} \\
&=&  2 \Big(\sum_{n} \| \sum_{\substack{P \in \P_{<J_n} \\ |I_P| \geq 2^k}} b_P \psi_{P,j}(x) \|^2_{V^r_{k_{n+1} \leq k < k_{n}}}  \Big)^{1/2}\\
&\le& 2\Big(\sum_{m, J \in \calJ_m}\| \sum_{\substack{P \in \P_{<J} \\ |I_P| \geq 2^k}} b_P \psi_{P,j}(x)\|_{V^r_k}^2 \Big)^{1/2} .
\end{eqnarray*}
\end{proof}

\begin{lemma} \label{l.E1(x)} It holds that
\begin{eqnarray*} 
\|E_1(x)\|_{L^2} &\lesssim_\eta& L A^{1 - \eta} \ .
\end{eqnarray*}
\end{lemma}

\begin{proof}
Using Cauchy-Schwartz  it suffices to show that for each $m$
\begin{eqnarray*}
\|\sum_{J \in \calJ_m: \ x \not\in J} \ \ \sum_{P \in \P_{<J}} |b_P \psi_{P,j}(x)|\|_{L^2} &\lesssim_{\eta}& L^{1/2} A^{-\eta}  \|b_P\|_{\ell^2(\bigcup_{J \in \calJ_m}\P_{<J})} \ \ .
\end{eqnarray*}
Using the Fefferman-Stein maximal inequality and the fact that the intervals in $\calJ_m$ are disjoint, it suffices to prove that if $J\in \calJ_m$ and  $x \not\in J$ then
\begin{eqnarray*}
\sum_{P \in \P_{<J}} |b_P \psi_{P,j}(x)| &\lesssim_{\eta}&  L^{1/2} A^{-\eta} \|b_P\|_{\ell^2(\P_{<J})} |J|^{-1/2} (\ma[1_J](x))^2.
\end{eqnarray*}
Now, if $T$ intersects $P_{<J}$ then $J\subset I_T$, therefore using \eqref{e.masq} we see that at most $L$ trees in $\T$ contribute a given $\P_{<J}$. Thus, using Cauchy Schwarz it suffices to show that, for each $T \in \T$,
\begin{eqnarray} \label{lrsetree}
\sum_{P \in T\cap \P_{<J}} |\psi_{P,j}(x)|^2 &\lesssim_{\eta}& A^{-2\eta} |J|^{-1} (\ma[1_J](x))^4.
\end{eqnarray}
Choosing $D_\eta$ large enough, the $(A,d)$ sparsity and \eqref{e.intcond} imply that for each $P$ in the sum above
\begin{eqnarray*}
\sup_{y \in I_P} \dis(y,\partial J) &\ge& D_\eta A^{-\eta} 2^{50A} |J|^{1/2}|I_P|^{1/2} \\
&\ge& 2^{49A}|J|^{1/2}|I_P|^{1/2} \ \ .
\end{eqnarray*}
Recall that $M$ is the order of the wave packet $\psi_{P,j}$. Thus, for $x \not\in J$ , choosing $M$ large enough we obtain
\begin{eqnarray*}
|\psi_{P,j}(x)|^2 &\lesssim_M& 2^{-49A(M-4)} (\frac{|J|}{|I_P|})^{-(M-4)} |I_P|^{-1} (\ma[1_{I_P}](x))^4 \\
&\lesssim_{\eta}& A^{-2\eta} (\frac{|J|}{|I_P|})^{-2} |J|^{-1} (\ma[1_{J}](x))^4.
\end{eqnarray*}
Summing over $P \in \P_{<J} \cap T$ we obtain \eqref{lrsetree}.
\end{proof}

\begin{lemma} \label{l.E2(x)} It holds that
\begin{eqnarray*} 
\|E_2(x)\|_{L^2_x} &\lesssim_\eta&  L^{2} A^{3 - \eta} 
\end{eqnarray*}
\end{lemma}

\begin{proof}
By Cauchy-Schwarz, it suffices to show that for $1 \leq m' < m$ we have 
\begin{eqnarray*}
\|\sum_{\substack{J \in \calJ_m \\ x \in J}} \sum_{\substack{P \in \P_{<J^{m'}} \\ |I_P| < |J|}} |b_P \psi_{P,j}(x)|\|_{L^2_x} &\lesssim_{\eta}& L^{1/2} A^{-\eta}  \|b_P\|_{\ell^2(\bigcup_{J \in \calJ_{m'}}\P_{<J})}
\end{eqnarray*}
The above bound will follow, by Cauchy-Schwarz, from the following two estimates
\begin{eqnarray} \label{lrsseeq1}
\|\sum_{\substack{J \in \calJ_m \\ x \in J}} \sum_{\substack{P \in \P_{<J^{m'}} \\ |I_P| < |J|}} |I_P|^{1/2}|\psi_{P,j}(x)|\|_{L^\infty} &\lesssim_\eta&   L A^{-2\eta} \\
\label{lrsseeq2}
\|\sum_{\substack{J \in \calJ_m \\ x \in J}} \sum_{\substack{P \in \P_{<J^{m'}} \\ |I_P| < |J|}} |I_P|^{-1/2} |b_P^2 \psi_{P,j}(x)|\|_{L^1} &\lesssim&  \|b_P\|^2_{\ell^2(\bigcup_{J \in \calJ_{m'}}\P_{<J})}.
\end{eqnarray}

To see \eqref{lrsseeq1} fix $x$ and choose the unique $J \in \calJ_m$ with $x \in J.$ As in Lemma \ref{l.E1(x)}, it suffices to show that for each $T \in \T$
\begin{eqnarray} \label{lrsseeq3}
\sum_{\substack{P \in \P_{<J^{m'}} \cap T: \ |I_P| < |J|}} |I_P|^{1/2}|\psi_{P,j}(x)| &\lesssim_\eta& A^{-2\eta}.  
\end{eqnarray}
Choosing $D_\eta$ large, it follows (as in Lemma \ref{l.E1(x)}) that the following holds for every $P$ in the sums above:
\begin{eqnarray*}
\inf_{y \in I_P} \dis(y,\partial J) &\ge& 2^{49A} |J|^{1/2} |I_P|^{1/2} \ \ .
\end{eqnarray*}
Since $|I_P| < |J|$ and  $P\in \P_{<J^{m'}}$, it follows that $I_P \cap J = \emptyset$. Since $x \in J$, using $(A,d)$ sparseness and \eqref{e.intcond} we obtain
\begin{eqnarray*}
\dis(x,I_P) &\ge& 2^{49A} |J|^{1/2} |I_P|^{1/2}  \ \ .
\end{eqnarray*} 
Due to the restriction of the sum to tiles in a single tree, each dyadic interval is the time interval of at most $O(1)$ tri-tiles, and so for each $k > 0$
\begin{eqnarray*}
\sum_{\substack{P \in \P_{<J^{m'}} \cap T:  \ \ |I_P| = 2^{-k}|J|}} |I_P|^{1/2}|\psi_{P,j}(x)| &\lesssim_M& 2^{-49A(M-1)} 2^{-k(M-1)/2} \  \  ,
\end{eqnarray*} 
and summing over $k$ gives \eqref{lrsseeq3}.

To see \eqref{lrsseeq2} simply use the fact that the intervals in $\calJ_m$ are pairwise disjoint to estimate the left side by
\begin{eqnarray*}
&\le& \|\sum_{P \in \bigcup_{J \in \calJ_{m'}}\P_{<J}} |I_P|^{-1/2} |b_P^2 \psi_{P,j}(x)|\|_{L^1} \quad \lesssim \quad \|b\|_{\ell^2(\bigcup_{J \in \calJ_{m'}} \P_{<J})}^2
\end{eqnarray*}
\end{proof}

Applying Lemma \ref{vnrmlemma} and Lemma~\ref{nonmaxbessellemma} as in the proof of \eqref{e.red41} we have
\begin{eqnarray*}
\|\sum_{m=1}^{n} \sum_{J \in \calJ_m} \sum_{P \in \P_{<J}} b_P \psi_{P,j}(x) \|_{L^2_x(V^r_n)} &\lesssim&   1 + \log^2(AL) 
\end{eqnarray*}
Thus, using Lemma~\ref{l.E1(x)E2(x)}, Lemma~\ref{l.E1(x)}, and Lemma~\ref{l.E2(x)}, to finish the proof of \eqref{e.red42} it suffices to establish  the following inequality (for each $m$ and $J \in \bigcup_{m}\calJ_L$):
\begin{eqnarray} \label{estwithbg}
\|\sum_{P \in \P_{<J} : \ |I_P| \geq 2^k} b_P \psi_{P,j}(x)\|_{L^2_x(V^r_k)} &\lesssim_{\epsilon}&  L^\epsilon \|b_P\|_{\ell^2(\P_{<J})}.
\end{eqnarray}
Let $\T_J$ be the collection of trees in $\T$ which contribute to $\P_{<J}.$ As above, we have $|\T_J| \leq L.$ 
For each $T \in \T_J$ let $\xi_T = c(\omega_{(P_{T})_j})$. Then, for each $P \in T \in \T_J$ we have 
\begin{eqnarray*}
\omega_{P_j} &\subset& (\xi_T - 10|e|\cdot|\omega_{P_j}|, \xi_T + 10|e|\cdot|\omega_{P_j}|).
\end{eqnarray*}
Furthermore, from condition \eqref{thirdstrongdisjoint} in the definition of strong j-disjointness and the fact that $J \subset I_T$ for each $T \in \T_J$, we have that 
\begin{eqnarray*}
\dis(\xi_T, \omega_{P_j}) &\ge& |\omega_{P_j}|/4
\end{eqnarray*}
for each $P \in \P_{<J}$ and each $T \in \T_J.$
Therefore, if we let 
\begin{eqnarray*}
R_k &=& \bigcup_{T \in \T_j} (\xi_T - 10|e|2^{-k}, \xi_T + 10|e|2^{-k})
\end{eqnarray*}
and $\Pi_k$ be the Fourier projection operator $\Pi_k[f] = \mathcal F^{-1}[1_{R_k}\widehat{f}]$
then, for each $k = -i \mod 4000$ we have
\begin{eqnarray} \label{projectionworks}
\sum_{\substack{P \in \P_{<J} \\ |I_P| \geq 2^k}} b_P \psi_{P,j} &=& \Pi_k[\sum_{P \in \P_{<J}} b_P \psi_{P,j}].
\end{eqnarray}
Thus, by Theorem~\ref{t.bourgain-varnorm} and Lemma \ref{nonmaxbessellemma} we have \eqref{estwithbg}.

\section{Concluding the proof of Theorem \ref{t.model-varnorm}} \label{s.long-concluding}

Let $\P$ be a finite subset of $\P_\nu$. Our aim is to prove that the trilinear form
\begin{eqnarray} \label{linearizedundualized}
\Lambda_{\P}(f_1,f_2,f_3) &=& \<\sum_{P \in \P} |I_P|^{-1/2}\<f_1,\phi_{P,1}\>\<f_2,\phi_{P,2}\>\phi_{P,3},f_3\>
\end{eqnarray}
satisfies restricted weak-type estimates with exponents $\alpha$ arbitrarily close to any vertex of $A$ define by \eqref{e.A-def}, with implicit constants  uniform over $\P$. We'll consider neighborhoods of $A_1(-\frac 1 2, \frac 1 2, 1)$, the other vertices could be treated similarly. 

By (dyadic) dilation symmetry we can assume $|F_1| \in [1/2, 1)$. Fix $s>1$ close to $1$ to be chosen later, and choose 
\begin{eqnarray*} B &=& \bigcup_{j=1}^3 \{\ma^s[1_{F_j}] \geq C |F_j|^{1/s}\} \ \ ,
\end{eqnarray*}
with $C$ sufficiently large so that $|B| \leq \frac{1}{4}$. Let $|f_1|\le 1_{F_1-B}$ and $|f_2|\le 1_{F_2}$ and $|f_3| \le 1_{F_3}$. Decompose $\P = \bigcup_{k \geq 0} \P_{k}$ where 
\begin{eqnarray*}
\P_{k} = \{P \in \P: 2^k \leq 1 + \dis(I_P,  B^c)/|I_P| < 2^{k+1} \}.
\end{eqnarray*}
For $P  \in \P_{k}$ we have
\begin{eqnarray*}
\sup_{I_P \subset I} \left( \frac{1}{|I|} \int |f_j(x)|^s \widetilde \chi_I(x)^2 \ dx\right)^{1/s} 
&\lesssim& 2^{k/s} \sup_{I_P \subset I} \left( \frac{1}{|2^kI|} \int |f_j(x)|^s \widetilde \chi_{2^kI}(x)^2\ dx\right)^{1/s} \\ 
&\lesssim&  2^{k/s} \inf_{x\in 2^k I_P}\ma^s[1_{F_j}](x) \\
&\lesssim& 2^{k/s} |F_j|^{1/s}
\end{eqnarray*}
Therefore, by Proposition \ref{p.sizebound}, for $j=2,3$ we have
\begin{eqnarray} \label{sizeboundeq}
S_j  \quad :=\quad \size_j(\P_{k},f_j) &\lesssim& 2^{k/s} |F_j|^{1/s}.
\end{eqnarray}
Here (and below) the implicit constants may depend on $r$, $s$, and $\beta_i$ (defined below). Now, when $j=1$ we will obtain the improved estimate 
\begin{eqnarray} \label{sizeboundeq2}
S_1 &:=& \size_1(\P_{k},1_{B^c}f_1) \leq C 2^{-(M-2)k} 
\end{eqnarray}
by exploiting the fact that the interval $I$ in the last sup has to be contained inside another $I_{P'}$ for some $P'\in \P_k$.

Now, applying Proposition \ref{l.size-increment} repeatedly, we obtain a decomposition of $\P_{k}$ into  collections of trees $(\T_n)_{n\in \mathbb Z}$ with
\begin{eqnarray} \label{treeboundinproof}
\sum_{T \in \T_n} |I_T| &\lesssim& 2^n,
\end{eqnarray}
such that  for any $T\in \T_n$ we have
\begin{eqnarray}\label{sizeboundinproof}
\size_{i}(T,f_i) &\lesssim& 2^{-n/(2s)}|F_i|^{1/(2s)}.
\end{eqnarray}

Now, for any tree $T$ we have
\begin{eqnarray} \label{treeestimate}
\sum_{P \in T} |I_P|^{-1/2} \prod_{i = 1}^3|\<f_i,\phi_{P,i}\>| 
&\le& 3 |I_T|\prod_{i = 1}^3 \size_i(T,f_i).
\end{eqnarray}
To see \eqref{treeestimate}, by further decomposing $T$ if needed we may assume that T is $i$-overlapping for some $i \in \{1, 2,3\}$.  Then estimating
\begin{eqnarray*}
|I_P|^{-1/2}|\<f_i,\phi_{P,i}\>| 
&\le& \size_i(T,f_i)
\end{eqnarray*}
and applying Cauchy-Schwarz to estimate the remaining bilinear sum by
\begin{eqnarray*}
&\lesssim& |I_T| \prod_{j \in \{1,2,3\} \setminus \{i\}}\size_j(T, f_j)
\end{eqnarray*}
one obtains \eqref{treeestimate}.

Applying \eqref{treeboundinproof}, \eqref{sizeboundinproof}, \eqref{treeestimate},  we obtain
\begin{eqnarray*} 
|\Lambda_k(f_1, f_2, f_3)|   &\lesssim& \sum_{n} 2^n \prod_{i=1}^3 \min(S_i, 2^{-n/(2s)}|F_i|^{1/(2s)}) \ \ , \\
 |\Lambda_k(f_1, f_2, f_3)| &:=& \sum_{P \in \P_{k}}|I_P|^{-1/2} \prod_{i = 1}^3|\<f_i,\phi_{P,i}\>|.
\end{eqnarray*}
For any $\beta_1,\beta_2,\beta_3\in [0,1]$, we obtain
\begin{eqnarray*}
|\Lambda_k(f_1, f_2, f_3)|   &\lesssim& S_1 S_2 S_3 \sum_{n} 2^n \min \Big(1, 2^{-n\frac{\beta_1+\beta_2 + \beta_3}{2s}} \prod_{i=1}^3  |F_i|^{\frac{\beta_i}{2s}} S_i^{-\beta_i} \Big) \ \ .
\end{eqnarray*}
The above estimate is a two sided geometric series if we choose $\beta_i$'s such that $\beta_1+\beta_2+\beta_3 > 2s$ (which is possible for $s$ close to $1$). Letting $\gamma_i := 2s\beta_i/(\beta_1+\beta_2+\beta_3)$ we obtain
\begin{eqnarray*}
|\Lambda_k(f_1, f_2, f_3)|   &\lesssim& \prod_{i=1}^3 S_i^{1-\gamma_i} |F_i|^{\gamma_i/(2s)}   \\
&\lesssim&  2^{\frac{k}{s}(2 - \gamma_2 - \gamma_3 - s(M-2)(1 - \gamma_1))}\Big(\prod_{i=1}^3 |F_i|^{1 -\frac{\gamma_i}{2}}\Big)^{1/s} \qquad \text{(using \eqref{sizeboundeq},  \eqref{sizeboundeq2})}.
\end{eqnarray*}

Again assuming that $\beta_1 + \beta_2 + \beta_3 > 2s$ we are guaranteed $\gamma_1 < 1$ and so, choosing $M$ large enough depending on $\beta$ we may sum in $k$ to conclude 
\begin{eqnarray*}
|\Lambda(f_1, f_2, f_3)| &\lesssim& \Big(\prod_{i=1}^3 |F_i|^{1 -\frac{\gamma_i}{2}}\Big)^{1/s}
\end{eqnarray*}
Since $|F_1| \sim 1$, we can ignore its contribution in the above estimate. Now, by sending $(s,\beta_1,\beta_2,\beta_3)$ to $(1,1,1,0)$ inside the region $\{\beta_1+\beta_2+\beta_3>2s\} \cap \{0\le \beta_1,\beta_2,\beta_3\le 1<s\}$, we obtain the desired claim.

\section{Proof of Theorem~\ref{t.model-square}}  \label{s.short-concluding}

The proof of Theorem~\ref{t.model-square} is entirely similar to the proof of Theorem~\ref{t.model-varnorm}, essentially the main difference is that variation-norm estimates such as the continuous L\'epingle inequality (see Lemma~\ref{l.lepingle}) is replaced by the classical Littlewood--Paley square function estimate. We briefly discuss the cosmetic changes, the details are left to the reader. We may define $\phi_{P,j}=\psi_{P,j}$ for $j=1,2$, and $\phi_{P,3}=\psi_{P,3} d_n$ if $|I_P|=2^n$ and $0$ otherwise. 

Now, the sizes are defined exactly as before, and to get the size estimates for $\size_3(\P,f)$ (as in Proposition~\ref{p.sizebound}) we use the same proof, the only difference is near the end we appeal to the classical $L^p$ estimates for the Littlewood--Paley square functions associated with scales of the underlying tree, instead of the continuous L\'epingle inequality. 

Now, to get the size increment estimate (as in Proposition~\ref{l.size-increment}) we use the same reduction to a Bessel inequality as in Theorem~\ref{t.vnbi}. To prove this Bessel estimate for the new $\phi_{P,3}$, we follow the same sequence of reductions and the proof reduces to proving Proposition~\ref{p.lastred} with the new modified wave packets. We perform the same partition of $\P$ as in \eqref{e.Ppartition}, and it suffices to show the following two analogues of \eqref{e.red41} and \eqref{e.red42}. Below we let $S_k$ denote the $\ell_2$ sum of a sequence indexed by $k$ and $(b_P)$ is a sequence on $\P$ with normalized $\ell^2(\P)$ norm.
\begin{eqnarray}
\label{e.red41-square}
\|\sum_{J \in \calJ} \sum_{P \in \P_J: \  |I_P| = 2^k} b_P \psi_{P,3}(x)\|_{L^2_x(S_k)}  &\lesssim_\eta&  1 + \log^2(AL) + A^{-\eta}L   \\
\label{e.red42-square}
\|\sum_{J \in \calJ} \sum_{P \in \P_{<J}: \  |I_P| = 2^k}  b_P \psi_{P,3}(x)\|_{L^2_x(S_k)}  &\lesssim_{\eta,\epsilon}&  (AL)^{\epsilon} + L^2A^{3 - \eta}   \ \ .
\end{eqnarray}
The proofs of these two estimates are similar. We'll use the same error terms $E(x)$, $E_1(x)$, and $E_2(x)$, and using analogues of Lemma~\ref{l.E(x)-source} and Lemma~\ref{l.E1(x)E2(x)} the proofs of \eqref{e.red41-square} and \eqref{e.red42-square} reduce to proving
\begin{eqnarray}
\label{e.sq-wavelet1}\|(\sum_{m=1} (\sum_{J\in \calJ_m} \sum_{P\in \P_J} b_P \psi_{P,3})^2)^{1/2}\|_2 &\lesssim& [1+\log(16A^2L)][1+ \log(L)] \\
\label{e.sq-wavelet2}\|(\sum_{m=1} (\sum_{J\in \calJ_m} \sum_{P\in \P_{<J}} b_P \psi_{P,3})^2)^{1/2}\|_2 &\lesssim&  [1+\log(16A^2L)][1+ \log(L)] \\
\label{e.sq-wavelet3}\|(\sum_{k} (\sum_{P \in \P_{<J}: \ |I_P|=2^k} b_P \psi_{P,3})^2)^{1/2}\|_2 &\lesssim_\epsilon& L^\epsilon \|b_P\|_{\ell^2(\P_{<J})}
\end{eqnarray}
We note that the square norm is bounded above by the $2$-variation norm. Thus, using Lemma~\ref{vnrmlemma} the estimates \eqref{e.sq-wavelet1} and \eqref{e.sq-wavelet2} follow from Lemma~\ref{nonmaxbessellemma}.  Similarly, using Proposition~\ref{bourgainrmprop} and the Fourier projection representation \eqref{projectionworks}, the estimate \eqref{e.sq-wavelet3} follows from Lemma~\ref{nonmaxbessellemma}.

\subsection*{Acknowledgement} This work was initiated while the authors were visiting the University of California, Los Angeles in Winter 2012, and the visit was supported in part by the AMS Math Research Communities program. The authors would like to thank the MRC and Christoph Thiele for their generous support, hospitality, and useful conversations.

\end{document}